\documentclass[12pt]{amsart}
\usepackage[scale=.86]{geometry}
\usepackage{geometry} 
\usepackage{graphicx}
\usepackage{bbm}
\usepackage{ulem}
\usepackage{amsthm}
\usepackage{mathrsfs}
\theoremstyle{plain}
\newtheorem{theorem}{Theorem}[section]
\newtheorem{lemma}{Lemma}[section]
\newtheorem{lem}{Lemma}[subsection]
\newtheorem{example}{Example}
\newtheorem{corollary}{Corollary}[section]
\newtheorem{proposition}{Proposition}[section]
\theoremstyle{definition}

\newtheorem*{remark}{Remark}
\newcommand{\ddr}{\mathrm{d}}
\newcommand{\etalchar}[1]{$^{#1}$}
\providecommand{\bysame}{\leavevmode\hbox to3em{\hrulefill}\thinspace}
\providecommand{\MR}{\relax\ifhmode\unskip\space\fi MR }

\providecommand{\href}[2]{#2}
\title[Logistic CSBPs: duality and reflection at infinity]{Continuous-state branching processes with competition:\\
duality and reflection at infinity}
\author{Cl\'ement Foucart}\thanks{foucart@math.univ-paris13.fr, Universit\'e Paris 13,  Laboratoire Analyse, G\'eom\'etrie \& Applications UMR 7539 Institut Galil\'ee} 
\usepackage[pagebackref,colorlinks,citecolor=Plum,urlcolor=Periwinkle,linkcolor=DarkOrchid]{hyperref}
\usepackage[dvipsnames]{xcolor}
\newcommand{\zmin}{\ensuremath{Z^{\text{min}}}}
\newcommand{\pmin}{\ensuremath{P^{\text{min}}}}
\newcommand{\paren}[1]{\ensuremath{\left( #1\right) }}
\newcommand{\imf}[2]{\ensuremath{#1\!\paren{#2}}}
\begin{document}

\begin{abstract}
The boundary behavior of continuous-state branching processes with quadratic competition is studied in whole generality. We first observe that despite competition, explosion can occur for certain branching mechanisms. We obtain a necessary and sufficient condition for $\infty$ to be accessible in terms of the branching mechanism and the competition parameter $c>0$. We show that when $\infty$ is inaccessible, it is always an entrance boundary. In the case where $\infty$ is accessible, explosion can occur either by a single jump to $\infty$ (the process at $z$ jumps to $\infty$  at rate $\lambda z$ for some $\lambda>0$) or by accumulation of large jumps over finite intervals.  We construct a natural extension of the minimal process and show that when $\infty$ is accessible and $0\leq \frac{2\lambda}{c}<1$, the extended process is reflected at  $\infty$. In the case $\frac{2\lambda}{c}\geq 1$, $\infty$ is an exit of the extended process. 
When the branching mechanism is not the Laplace exponent of a subordinator, we show that the process with reflection at $\infty$ get extinct almost-surely. Moreover absorption at $0$ is almost-sure if and only if Grey's condition is satisfied. 
When the branching mechanism is the Laplace exponent of a subordinator, necessary and sufficient conditions are given for a stationary distribution to exist. The Laplace transform of the latter is provided.  The study is based on classical time-change arguments and on a new duality method relating logistic CSBPs  with certain generalized Feller diffusions.
\end{abstract}
%

%

\maketitle
\section{Introduction}
Continuous-state branching processes (CSBPs for short) have been defined by Ji\v{r}ina \cite{Jirina} and Lamperti \cite{MR0208685} for modelling the size of a random continuous population whose individuals reproduce and die independently with the same law. Lamperti \cite{MR0217893} and Grimvall \cite{MR0362529} have shown that these processes arise as scaling limits of Galton-Watson Markov chains. Their laws are characterised in terms of a L\'evy-Khintchine function $\Psi$ (called a branching mechanism). A shortcoming of CSBPs for modelling population lies in their degenerate longterm behavior. In the long run, a CSBP either tends to $0$ or to $\infty$. On the event of extinction, the process can decay indefinitely or be absorbed at $0$ in finite time. Similarly, on the event of non-extinction, the CSBP can grow indefinitely or be absorbed at $\infty$ in finite time. The latter event is called explosion and occur typically when the process performs infinitely many large jumps in a finite time with positive probability.  Since the sixties, several generalizations of CSBPs have been defined to overcome various unrealistic properties of pure branching processes. Lambert \cite{MR2134113} has introduced a generalization of these processes by incorporating pairwise interactions between individuals. These processes, called logistic continuous-state branching processes, are the random analogues of the logistic equation 
\begin{equation}\label{deterministic}\ddr z_t=\gamma z_t\ddr t-\frac{c}{2}z^2_t\ddr t,
\end{equation}
where, informally speaking, the Malthusian growth $\gamma z_t \ddr t$ is replaced by the full dynamics of a continuous-state branching process. For instance, when the mechanism $\Psi$ of the CSBP reduces to $\Psi(z)=\frac{\sigma^2}{2}z^2-\gamma z$, the process $(Z_t,t\geq 0)$ is the logistic Feller diffusion
\begin{equation}\label{continuouscbc} \ddr Z_t=\sigma\sqrt{Z_t}\ddr B_t+\gamma Z_t\ddr t-\frac{c}{2}Z_t^2\ddr t. 
\end{equation}

The negative quadratic drift represents additional deaths occurring due to \textit{pairwise fights} among individuals. Intuitively, these fights  can be interpreted as competition (for resources for instance). We refer to Le, Pardoux and Wakolbinger \cite{MR3034788} and Berestycki, Fittipaldi and Fontbona \cite{Berestycki2017} for a study of the competition at the level of the genealogy. In a logistic CSBP, individuals and their progenies are not independent between each other, and the branching property, from which all properties of CSBPs can be deduced, is lost. One of the main interest of logistic CSBPs is to provide a model of population with a possible self-limiting growth. The objective of this article is to study these processes with most general mechanisms and to understand precisely how the competition regulates the growth. We shall study the nature of the boundaries $0$ (extinction of the population) and $\infty$ (explosion of the population).  

Throughout this article, we follow the terminology of Feller, introduced in \cite{MR0063607}, for classifying boundaries of a diffusion (see Section \ref{preliminaries} for their meaning). The state of the art is as follows. In the continuous case (\ref{continuouscbc}), Feller tests provide that $\infty$ is an entrance and $0$ an exit. For a general mechanism $\Psi$, the logistic CSBP has (typically unbounded) positive jumps and such general tests do not exist. Lambert in \cite{MR2134113} has found a set of sufficient conditions over the mechanism $\Psi$ for $\infty$ to be an entrance boundary. Under these conditions, it is also shown in \cite{MR2134113} that the competition alone has no impact on the extinction of the population. In other words, the boundary $0$ is an exit if and only if the branching mechanism $\Psi$ satisfies Grey's condition. The entrance property from $\infty$ coincides with the notion of coming down from infinity observed in many stochastic models. We refer for instance to Cattiaux et al. \cite{cattiaux2009} and Li \cite{LI2018} for recent related works and to Donnelly \cite{MR1112408} for a classical result in coalescent theory.
We shall follow a different route than \cite{MR2134113} and study directly the semigroup of logistic CSBPs. A rather surprising first phenomenon is that the competition does not always prevent explosion. Some reproduction laws have large enough tails for $\infty$ to be accessible (meaning for explosion to occur). We provide a necessary and sufficient condition for $\infty$ to be inaccessible and show that under this condition the boundary $\infty$ is an entrance. Since the competition pressure increases with the size of the population, one may wonder if some compensation occurs near the boundary $\infty$ for a general mechanism $\Psi$. The main contribution of this article is to answer the following question. 
\textit{Is it possible for a logistic continuous-state branching process to leave and return to $\infty$ in finite time?} We shall indeed see that the reproduction can be strong enough for explosion to occur and the quadratic competition strong enough to instantaneously push the population size back into $[0,\infty)$ after explosion. This phenomenon is captured by the notion of regular \textit{instantaneous reflecting} boundary. By reflecting, we mean here that the Lebesgue measure of $\{t\geq 0, Z_t=\infty\}$ is zero almost-surely. The boundary is instantaneous in the sense that starting from $\infty$ the process enters immediately $(0,\infty)$.  Only in some cases, for which explosion is made by a single jump, the boundary $\infty$ is an exit. We stress also that it may well occur that the population goes extinct after exploding, so that $\infty$ is not always recurrent. 

In order to classify the boundaries as explained above, we need to define an extension of the minimal process in $[0,\infty]$. This requires in general a deep study of the minimal semi-group. However, since processes with competition do not satisfy the branching property, most arguments for CSBPs are not applicable. The resolvent of logistic CSBPs is rather involved and we will not discuss all possibilities of extensions in this article but only construct a natural one by approximation. We first establish a \textit{duality relationship} between non-explosive logistic continuous-state branching processes and some generalizations of the logistic Feller diffusion process (\ref{continuouscbc}). Namely we will show that when  $\infty$ is inaccessible for the process $(Z_t,t\geq 0)$,
then for any $x\geq 0$, $z\in [0,\infty[$ and $t\geq 0$,
$$(\star) \qquad \mathbb{E}_z(e^{-xZ_t})=\mathbb{E}_x(e^{-zU_t})$$
where $(U_t, t\geq 0)$ is a solution to
\begin{equation}\label{GFDi}\ddr U_t=\sqrt{cU_t}\ddr B_t-\Psi(U_t)\ddr t, \quad U_0=x.\end{equation}
We shall see that the condition for $\infty$ to be inaccessible for $(Z_t,t\geq 0)$ is precisely given by Feller's test for $0$ to be an exit of $(U_t, t\geq 0)$. 
We stress that Equation (\ref{GFDi}) has not always a unique solution as $0$ can be regular for certain non-lipschitz mechanisms $\Psi$. It is precisely for such mechanisms that $\infty$ will be regular for logistic CSBPs. 
Heuristically, if $(\star)$ holds for some processes $(Z_t,t\geq 0)$ and $(U_t,t\geq 0)$, then the entrance boundaries of $(Z_t,t\geq 0)$ will be classified in terms of exit boundaries of $(U_t,t\geq 0)$. We refer to Cox and R\"osler \cite{MR724061} and Liggett \cite{MR2108619} for a  study of boundaries by duality of semi-groups. The identity $(\star)$ provides a representation of the semi-group of any non-explosive process with competition and  will allow us to construct an extended process over $[0,\infty]$ with $\infty$ reflecting \textit{as a limit of non-explosive processes}. We highlight that this construction is different from the classical It\^o's concatenation procedure for building recurrent extensions. 
In particular, our approach is not based on a measure theoretical description of the excursions from $\infty$ but on a direct description of the extended semi-group.  \\

A very similar phenomenon of reflection at $\infty$ has been recently observed by Kyprianou et al. \cite{kyprianou2017} for a certain exchangeable fragmentation-coalescence process.  We shall observe the same phase transition between the reflecting boundary case and the exit boundary one. Lambert \cite{MR2134113} and Berestycki \cite{MR2110018} have noticed that discrete logistic branching processes share many properties with the number of fragments in some exchangeable coalescence-fragmentation processes. Discrete logistic branching processes are interesting in their own and will be studied elsewhere.  We highlight that contrary to the process studied in  \cite{kyprianou2017}, a logistic CSBP can reach $\infty$ by accumulation of large jumps over a finite interval of time. 
We mention that the duality $(\star)$ has been observed in a spatial context for the branching mechanism $\Psi(u)=\frac{\sigma^2}{2}u^2-\gamma u$ by Horridge and Tribe \cite{MR2096217} for the logistic SPDE, see also Hutzenthaler and Wakolbinger \cite{MR2308333}. Lastly, other competition mechanisms than the quadratic drift have been studied. We refer for instance to the monograph of Pardoux \cite{MR3496029} and Ba and Pardoux \cite{ba2015} for some generalisations of the logistic Feller diffusions. 
It is worth noticing that the relation $(\star)$ does not hold for general competition mechanims. 
\\
 
The paper is organised as follows. In Section \ref{preliminaries}, we recall some known facts about CSBPs and define minimal logistic CSBPs through a martingale problem. We state our main results in Section \ref{mainresults} and describe some examples. In Section \ref{existence}, we show how to solve the martingale problem by time-changing an Ornstein-Uhlenbeck type process. Some first properties of the minimal process, such as a criterion for its explosion, are derived from this time-change. In Section \ref{entrancesection}, we gather the possible behaviors of the diffusion $(U_t,t\geq 0)$ at its boundaries. Then we establish the duality under the non-explosion assumption and deduce the entrance property. In Section \ref{regularsection}, we define and study an extension of the minimal process. Lastly, in Appendix, we provide the calculations  needed for classifying the boundaries of $(U_t,t\geq 0)$ according to $\Psi$ and the parameter $c$.
\section{Preliminaries}\label{preliminaries}
\noindent As we will use Feller's terminology repeatedly, we briefly recall how to classify boundaries. Consider a process valued in an interval $(a,b)$ with $a,b\in \bar{\mathbb{R}}$ and $a<b$,
\begin{itemize}
\item[-] the boundary $b$ is said to be \textit{accessible} if there is a positive probability that it will be reached in finite time (the process can enter into $b$). If $b$ is accessible, then 
either the process cannot get out from $b$ and the boundary $b$ is said to be an \textit{exit}
or the process can get out from $b$ and the boundary $b$ is called a \textit{regular} boundary. 
\item[-] If the boundary $b$ is \textit{inaccessible}, then 
either the process cannot get out from $b$, and the boundary $b$ is said to be \textit{natural}
or
the process can get out from $b$ and the boundary $b$ is said to be an \textit{entrance}.\\
\end{itemize}
\textbf{Notation.} We denote by $[0,\infty]$ the extended half-line and by $C_b([0,\infty])$ the space of continuous real-valued functions defined over $[0,\infty]$. Since $[0,\infty]$ is compact, any function $f\in C_b(|0,\infty])$ is bounded. We set $\mathcal{D}([0,\infty])$ the space of c\`adl\`ag functions from $\mathbb{R}_+$ to $[0,\infty]$. For any interval $I\subset \mathbb{R}$, we denote by $C_c^{2}(I)$ the space of continuous functions over $I$ with compact support that have continuous first two derivatives.
\\

We recall the definition and some basic properties of continuous-state branching processes without competition. 
Most of the sequel can be found in Chapter 12 of Kyprianou's book \cite{MR3155252}. 
A CSBP is a Feller process $(X_t,t\geq 0)$ valued in $[0,\infty]$ satisfying the branching property: for any $z,z'\geq 0$, $t\geq0$ and $x>0$
\begin{equation*}\label{branching}
\mathbb{E}_{z+z'}[e^{-xX_t}]=\mathbb{E}_{z}[e^{-xX_t}]\mathbb{E}_{z'}[e^{-xX_t}].
\end{equation*}
The branching and Markov properties ensure the existence of a map $(x,t)\mapsto u_{t}(x)$ such that for all $x>0$ and all $t,s\geq 0$, $u_t(x)\in (0,\infty)$,
\begin{equation}\label{cumulant}
\mathbb{E}_z[e^{-xX_{t}}]=e^{-zu_{t}(x)}  \text{ and } u_{s+t}(x)=u_{s}\circ u_t(x).
\end{equation}
Silverstein in \cite[Theorem 4, page 1046]{MR0226734} has shown that the map $t\mapsto u_{t}(x)$ is the unique solution to a non-linear ordinary differential equation 
\begin{equation}\label{cumulantode} \frac{\ddr}{\ddr t}u_t(x)=-\Psi(u_t(x)) \quad \text{for all } x \in (0,\infty)
\end{equation}
where $\Psi$ is a L\'evy-Khintchine function of the form
\begin{equation} \label{Levykhintchine} \Psi(z)=-\lambda+\frac{\sigma^{2}}{2}z^{2}+\gamma z+\int_{0}^{+\infty}\left(e^{-zx}-1+zx\mathbbm{1}_{\{x\leq 1\}}\right)\pi(\ddr x)
\end{equation}
with $\lambda\geq 0, \gamma\in\mathbb{R}$, $\sigma \geq 0$, and $\pi$ a Borel measure carried on $\mathbb{R}_{+}$ satisfying
\begin{equation*}
\int_{0}^{+\infty} (1\wedge x^{2}) \pi(\ddr x)<+\infty.
\end{equation*}
Any branching mechanism $\Psi$ is Lipschitz on compact subsets of $(0,\infty)$ and thus the deterministic equation (\ref{cumulantode}) admits a unique solution. As in Silverstein \cite{MR0226734}, we interpret the killing term with parameter $\lambda$ as the possibility for the process to jump to $\infty$ in finite time. 
Since for any $t\geq 0$ and any $x>0$, $u_t(x)>0$ then according to the semi-group equation (\ref{cumulant}), $\infty$ and $0$ are either natural or exit boundaries. Grey \cite{MR0408016} classifies further the boundaries $\infty$ and $0$ of a CSBP as follows.
\begin{itemize}
\item[-] The boundary $\infty$ is accessible if and only if $\int_{0+}\frac{\ddr u}{|\Psi(u)|}<+\infty$. 
\item[-] The boundary $0$ is accessible if and only if $\int^{\infty}\frac{\ddr u}{\Psi(u)}<\infty$ (Grey's condition)
\end{itemize}
The integral conditions above ensure respectively the existence of a non-degenerate solution of (\ref{cumulantode}) started from $x=0+$ and $x=\infty$. 
It is important to note that $\lambda=0$ is necessary for $\infty$ to be inaccessible but not sufficient. Indeed, the process can \textit{explode continuously} by having unbounded paths over finite time intervals. A basic example is provided by the stable mechanism $\Psi(z)=-z^{\alpha}$ for $\alpha \in (0,1)$ which satisfies $\int_{0+}\frac{\ddr u}{|\Psi(u)|}<\infty$. 
We now recall the longterm behavior of CSBPs. We refer to Theorem 12.5 in \cite{MR3155252} for a complete classification. Denote by $\rho$ the largest positive root of $\Psi$.  For any $z\in [0,\infty]$, \[\mathbb{P}_z(X_t \underset{t\rightarrow \infty}{\longrightarrow} 0)=e^{-z\rho} \text{ and } \mathbb{P}_z(X_t \underset{t\rightarrow \infty}{\longrightarrow} \infty)=1-e^{-z\rho}.\]
When $-\Psi$ is the Laplace exponent of a subordinator then $\rho=\infty$ and the process goes to $\infty$ almost-surely. When $-\Psi$ is not the Laplace exponent of a subordinator then $\rho<\infty$, and the process goes to $0$ with positive probability. If moreover $\int^{\infty}\frac{\ddr u}{\Psi(u)}=\infty$ then $X_t \underset{t\rightarrow \infty}{\longrightarrow} 0$ with positive probability albeit $X_t>0$ for all $t\geq 0$ almost-surely. In the latter case, we say that $0$ is \textit{attracting}. A classical construction of a CSBP with mechanism $\Psi$ is by time-changing a spectrally positive L\'evy process with Laplace exponent $-\Psi$ (see for instance Lamperti \cite{MR0208685}, Caballero, Lambert and Uribe-Bravo \cite{MR2592395}). In particular, the sample paths of a c\`adl\`ag CSBP have no negative jump and are non-decreasing when $-\Psi$ is the Laplace exponent of a subordinator. This time-change leads to the following form for the extended generator of $(X_t,t\geq 0)$. For any $f\in \mathcal{C}^2_c((0,\infty))$ \footnote{the space of twice continuously differentiable functions vanishing outside a compact subset of $(0,\infty)$.}
\begin{equation*}
\mathscr{G}f(z):=
-\lambda zf(z)+\frac{\sigma^{2}}{2}zf''(z)-\gamma zf'(z)+z\int_{0}^{\infty}(f(z+u)-f(z)-u1_{[0,1]}(u)f'(z))\pi(\ddr u).
\end{equation*}To incorporate quadratic competition, one considers an additional negative quadratic drift  in the extended generator above and set \begin{equation}\label{generatormin}
\mathscr{L}f(z):=\mathscr{G}f(z)-\frac{c}{2}z^2f'(z).
\end{equation}
We define a \textit{minimal} logistic continuous-state branching process as a c\`adl\`ag Markov process $(\zmin_t,t\geq 0)$  on $[0,\infty]$ with $0$ and $\infty$ absorbing, satisfying the following martingale problem \textbf{(MP)}. For any function $f\in C^2_c((0,\infty))$, the process 
\begin{equation}\label{MP}
 t \in [0,\zeta)\mapsto f(\zmin_t)-\int_0^t \mathscr{L} f(\zmin_s)\, \ddr s 
\end{equation}
is a martingale under each $\mathbb{P}_z$,   
with $\zeta:=\inf\{t\geq 0; Z_t\notin (0,\infty)\}$. 
By minimal process, we mean that the process remains at $\infty$ from its first (and only) explosion time $\zeta_\infty:=\inf\{t\geq 0, \zmin_t=\infty\}$.  As already observed by Lambert \cite{MR2134113}, one way to construct a minimal logistic CSBP is by time-changing an Ornstein-Uhlenbeck type process. The problem of explosion is not discussed in \cite{MR2134113} and we shall give out some details in Section \ref{existence}. In the sequel, we say that a process $(Z_t,t\geq 0)$ extends the minimal process if $(Z_t,t\geq 0)$ takes its values in $[0,\infty]$ and $(Z_{t\wedge \zeta_\infty},t\geq 0)\overset{\mathcal{L}}{=}(\zmin_t,t\geq 0)$. Note that elementary return processes restarting after explosion from states in $(0,\infty)$ are ruled out from our study. We will only be interested in the existence of a continuous extension $(Z_t,t\geq 0)$, for which $Z_t\underset{t\rightarrow 0}{\longrightarrow} \infty$, $\mathbb{P}_\infty$-almost-surely.  As explained in the introduction, the semi-group of logistic CSBPs will be represented in terms of a certain diffusion. For any mechanism $\Psi$ of the form (\ref{Levykhintchine}), we call $\Psi$-generalized Feller diffusion, the minimal diffusion $(U_t,t<\tau)$  solving
\begin{equation}\label{GFD}\ddr U_t=\sqrt{cU_t}\ddr B_t-\Psi(U_t)\ddr t, \quad U_0=x
\end{equation}
where $(B_t, t\geq 0)$ is a Brownian motion and $\tau:=\inf\{t; U_t\notin (0,\infty)\}$. As $u\mapsto \sqrt{u}$ is $1/2$-H\"older and $\Psi$ is locally Lipschitz, standard results (see e.g. \cite[Section 3, Chapter IX]{MR1725357}) ensure the existence and uniqueness of a strong solution to Equation (\ref{GFD}) up to time $\tau$. Note that (\ref{GFD}) coincides with (\ref{cumulantode}) when $c=0$. The duality between the $\Psi$-generalized Feller diffusion and the logistic CSBP can be easily seen at the level of generators, see the forthcoming Lemma \ref{generatordualitylemma}. Duality of semigroups requires more work and is part of the main results.
\section{Main results}\label{mainresults}
\begin{theorem}[Accessibility of $\infty$]\label{explosion} Assume $c>0$. The boundary $\infty$ is inaccessible for $(\zmin_t,t\geq 0)$ if and only if $$\mathcal{E}:=\int_{0}^{\theta}\frac{1}{x}\exp {\left(\frac{2}{c}\int_x^\theta \frac{\imf{\Psi}{u}}{u}\, \ddr u\right)}\, \ddr x=\infty,
\text{ for some (and then for all) }\theta >0.$$
\end{theorem}
\begin{remark} The integrals $\int_0 \frac{|\Psi(u)|}{u}\ddr u$ and $\int^{\infty} \log(u) \pi(\ddr u)$ have the same nature. In particular $\int^{\infty} \log(u) \pi(\ddr u)<\infty$ implies $\mathcal{E}=\infty$.
\end{remark}
The next theorems introduce extensions in $[0,\infty]$ of the minimal process. The usual convention $0\cdot\infty=\infty\cdot 0=0$ is taken. In particular, note that $e^{-0.z}=1$ for all $z\in [0,\infty]$.
\begin{theorem}[Infinity as entrance boundary]\label{entranceboundarytheo} Assume $\mathcal{E}=\infty$. 
The Markov process $(\zmin_t,t\geq 0)$ can be extended in $[0,\infty]$ to a Feller process $(Z_t,t\geq 0)$ with  $\infty$ as an entrance boundary. The boundary $0$ is an exit of the diffusion $(U_t,t\geq 0)$  solution to (\ref{GFD}), and the semi-group of $(Z_t,t\geq 0)$ satisfies for all $t\geq 0$, all $z\in [0,\infty]$, all $x\in [0,\infty)$ \[\mathbb{E}_z(e^{-xZ_t})=\mathbb{E}_x(e^{-zU_t}).\]
\end{theorem}
For any L\'evy measure $\pi$ and any $x\geq 0$, set $\bar{\pi}(x):=\pi([x,\infty))$. Given $\Psi$ of the form (\ref{Levykhintchine}) and $k\geq 1$, define a branching mechanism $\Psi_{k}$ by
$$\Psi_{k}(z):=\frac{\sigma^2}{2}z^2+\gamma z+\int_{0}^{\infty}\left(e^{-zx}-1+zx\mathbbm{1}_{x\in (0,1)}\right)\pi_{k}(\ddr x),\text{ with }\pi_k=\pi_{|]0,k[}+(\bar{\pi}(k)+\lambda)\delta_k.$$
Plainly $|\Psi_k'(0+)|<\infty$, for any $k\geq 1$, and by Theorem \ref{explosion}, the minimal logistic CSBP with mechanism $\Psi_k$ does not explode. Call $(Z_t^{(k)}, t\geq 0)$ the c\`adl\`ag logistic CSBP, provided by Theorem \ref{entranceboundarytheo}, with mechanism $\Psi_{k}$ and $\infty$ as entrance boundary.
\begin{theorem}[Infinity as regular reflecting boundary]\label{regularboundarytheo} Assume $\mathcal{E}<\infty$ and  $0\leq \frac{2\lambda}{c}<1$. 
The processes $(Z_t^{(k)},t\geq 0)$ converges weakly in $\mathcal{D}([0,\infty])$
towards a Feller process $(Z_t, t\geq 0)$, extending $(\zmin_t,t\geq 0)$, with $\infty$ regular instantaneous reflecting.
The semi-group of $(Z_t, t\geq 0)$ satisfies for all $t\geq 0$, all $z\in [0,\infty]$ and $x\in [0,\infty)$,
\[\mathbb{E}_{z}(e^{-xZ_t})=\mathbb{E}_x(e^{-zU^{0}_t})\]
where $(U^{0}_t, t\geq 0)$ is solution to (\ref{GFD}) with $0$ regular absorbing. 
\end{theorem}
\begin{theorem}[Infinity as exit boundary]\label{exitboundarytheo} Assume $\frac{2\lambda}{c}\geq 1$ (so that $\mathcal{E}<\infty$). The processes $(Z_t^{(k)},t\geq 0)$ converges weakly in $\mathcal{D}([0,\infty])$ towards a Feller process $(Z_t, t\geq 0)$, extending $(\zmin_t,t\geq 0)$, with $\infty$ an exit. The boundary $0$ is an entrance of the diffusion $(U_t,t\geq 0)$  solution to (\ref{GFD}) and the semi-group of $(Z_t, t\geq 0)$ satisfies for all $t\geq 0$, all $z\in [0,\infty]$ and $x\in (0,\infty)$,
\[\mathbb{E}_{z}(e^{-xZ_t})=\mathbb{E}_x(e^{-zU_t}).\]
\end{theorem}
\begin{figure}[h!]
\centering \noindent
\includegraphics[height=.16 \textheight]{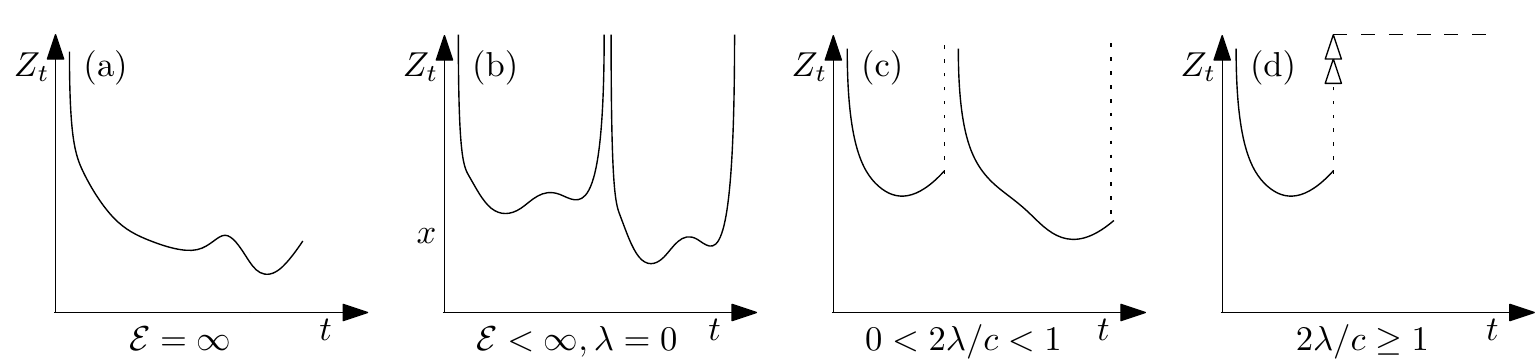}
\caption{Symbolic representation of the four behaviors at $\infty$.}
\label{behaviorcbc}
\end{figure}
\begin{corollary}[Zero as exit or natural boundary]\label{zeroboundary}\
\begin{itemize}
\item[i)] Assume $\int^{\infty}\frac{\ddr z}{|\Psi(z)|}<\infty$ then $0$ is an exit boundary of $(Z_t,t\geq 0)$ and $\infty$ is an entrance boundary of $(U_t,t\geq 0)$. 
\item[ii)] Assume $\int^{\infty}\frac{\ddr z}{|\Psi(z)|}=\infty$ then $0$ is a natural boundary of $(Z_t,t\geq 0)$ and $\infty$ is a natural boundary of $(U_t,t\geq 0)$.
\end{itemize}
\end{corollary}
The boundary behaviors found in Theorems \ref{entranceboundarytheo}, \ref{regularboundarytheo} and \ref{exitboundarytheo} and Corollary \ref{zeroboundary} can be summarized as follows
\begin{table}[htpb]
\begin{center}
\begin{tabular}{c|l}
Boundary of $Z$ &  Boundary  of $U$ \\
\hline
$\infty$  entrance  &  $0$ exit \\
\hline
$\infty$  regular reflecting  &  $0$ regular absorbing\\
\hline
$\infty$  exit & $0$ entrance\\ 
\hline
$0$  exit & $\infty$ entrance\\ 
\hline
$0$ natural & $\infty$ natural\\
\hline
\end{tabular}
\vspace*{4mm}
\caption{Boundaries of $Z$ and boundaries of $U$}
\label{correspondance}
\end{center}
\end{table}
\vspace*{-0.5cm}
\begin{corollary}[Stationarity]\label{StationarityZ}
Assume $\Psi(z)<0$ for all $z>0$, then $-\Psi$ is the Laplace exponent of a subordinator and takes the form
\begin{equation*} \Psi(z)=-\lambda-\delta z-\int_{0}^{\infty}(1-e^{-zu})\pi(\ddr u)
\end{equation*}
with $\lambda\geq 0$, $\delta\geq 0$ and $\int_{0}^{\infty}(1\wedge u)\pi(\ddr u)<\infty$. Assume $0\leq\frac{2\lambda}{c}<1$ and
define the condition \textbf{(A)} as follows 
\[\textbf{(A)}: (\delta=0 \text{ and } \bar{\pi}(0)+\lambda\leq c/2).\]
\begin{itemize}
\item[-] If $\textbf{(A)}$ is satisfied then  $(Z_t, t\geq 0)$ converges in probability to $0$.
\item[-] If $\textbf{(A)}$ is not satisfied then $(Z_t,t\geq 0)$ converges in law towards the distribution carried over $(\frac{2\delta}{c},\infty)$ whose Laplace transform is 
$$x\in \mathbb{R}_+\mapsto \mathbb{E}[e^{-xZ_\infty}]:=\frac{\int_{x}^{\infty}\exp\left(\int_{\theta}^{y}\frac{2\Psi(z)}{cz}\ddr z\right)\ddr y}{\int_{0}^{\infty}\exp\left(\int_{\theta}^{y}\frac{2\Psi(z)}{cz}\ddr z\right)\ddr y}.$$
\end{itemize}
\end{corollary}

\begin{remark}  The condition in Corollary \ref{StationarityZ} for the existence of a non-degenerate stationary distribution can be rephrased as follows. The condition \textbf{(A)} is not satisfied  if and only if at least one of the following holds
\begin{center}
 $\underset{u\rightarrow \infty}{\lim}\frac{\Psi(u)}{u}=-\delta\neq 0$, $\pi((0,1))=\infty$, $\bar{\pi}(0)+\lambda>\frac{c}{2}$.
\end{center}
This already appears with $\lambda=0$ and the log-moment assumption in \cite[Theorem 3.4]{MR2134113}. One can easily see from the Laplace transform that $\lambda=0$ and $\int^{\infty}\log (x) \pi(\ddr x)<\infty$ are necessary and sufficient conditions for the stationary distribution to admit a first moment. 
\end{remark}
\newpage
\begin{theorem}[Long-term behaviors] \label{longtermtheo}\
Consider $(Z_t,t\geq 0)$ the process started from $z\in (0,\infty)$.
\begin{itemize}
\item[1)] If $0\leq \frac{2\lambda}{c}<1$ and $\Psi(z)\geq 0$ for some $z>0$  then
\begin{itemize}
\item[1-1)] If $\int^{\infty}\frac{\ddr u}{\Psi(u)}=\infty$, then $Z_t>0$ for any $t\geq 0$ a.s. and $Z_t\underset{t\rightarrow \infty}{\longrightarrow} 0$ a.s.
\item[1-2)] If $\int^{\infty}\frac{\ddr u}{\Psi(u)}<\infty$, then $(Z_t,t\geq 0)$ will be absorbed at $0$ in finite time almost-surely.
\end{itemize}
\item[2)] If $\frac{2\lambda}{c}\geq 1$ and $\Psi(z)<0$ for all $z>0$ then $(Z_t,t\geq 0)$ will be absorbed at $\infty$ in finite time almost-surely.
\item[3)] If $\frac{2\lambda}{c}\geq 1$ and $\Psi(z)\geq 0$ for some $z>0$ then
$$\mathbb{P}_z(Z_t\underset{t\rightarrow \infty}{\longrightarrow} 0)=1-\mathbb{P}_z(\zeta_\infty<\infty)=\frac{\int_{0}^{\infty}e^{-zu}\frac{1}{u}\exp\left(-\int_{\theta}^{u}\frac{2\Psi(v)}{cv}\ddr v\right)\ddr u}{\int_{0}^{\infty}\frac{1}{u}\exp\left(-\int_{\theta}^{u}\frac{2\Psi(v)}{cv}\ddr v\right)\ddr u}\in (0,1).$$
If $\int^{\infty}\frac{\ddr u}{\Psi(u)}=\infty$, then $Z_t>0$ for any $t\geq 0$ a.s and if $\int^{\infty}\frac{\ddr u}{\Psi(u)}<\infty$, then $\{Z_t\underset{t\rightarrow \infty}{\longrightarrow} 0\}=\{\zeta_0<\zeta_\infty\}$.
\end{itemize}
\end{theorem}
\begin{figure}[h!]
\centering \noindent
\includegraphics[height=.16 \textheight]{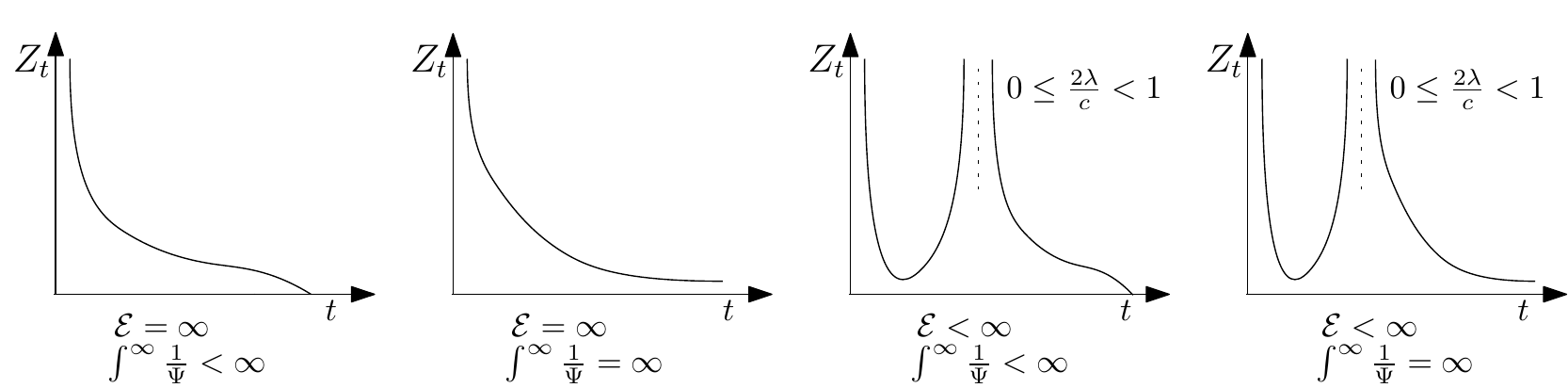}
\caption{Symbolic representation of the two behaviors at $0$ in the non-subordinator case with $\infty$ entrance or reflecting.}
\label{behaviorcbc0}
\end{figure}
We provide now several examples for which different behaviors at infinity occur.
\begin{example}  Consider $\alpha \in (0,2]$, $\alpha\neq 1$ and $\Psi(z)=(\alpha-1)z^{\alpha}$. Since $\int_{0}\frac{|\Psi(z)|}{z}\ddr z<\infty$, then $\mathcal{E}=\infty$ and $\infty$ is an entrance boundary (case (a) in Figure \ref{behaviorcbc}). For any $t\geq 0$, $z\in [0,\infty]$ and $x\in [0,\infty[$ 
$$\mathbb{E}_z(e^{-xZ_t})=\mathbb{E}_x(e^{-zU_t})\text{ with }\ddr U_t=\sqrt{cU_t}\ddr B_t+(1-\alpha)U_t^{\alpha}\ddr t, \ U_0=x,$$
the boundary $0$ of $(U_t,t\geq 0)$ being an exit. Note that when $\alpha \in (0,1)$, the CSBP without competition explodes, so that here  competition prevents explosion. 
\end{example}
\begin{example} Let $\lambda>0$ and $\pi \equiv 0$ in order that $\Psi(x)=-\lambda$ for all $x\geq 0$. If $\frac{2\lambda}{c}<1$ then $\infty$ is regular reflecting (case (c) in Figure \ref{behaviorcbc}).  For any $t\geq 0$, $z\in [0,\infty]$ and $x\in [0,\infty)$ 
$$\mathbb{E}_z(e^{-xZ_t})=\mathbb{E}_x(e^{-zU^0_t})\text{ with } \ddr U^0_t=\sqrt{cU^0_t}\ddr B_t+\lambda \ddr t, \ U^0_0=x,$$
the boundary $0$ of $(U^0_t,t\geq 0)$ being regular absorbing. If $\frac{2\lambda}{c}\geq 1$ then $\infty$ is an exit (case (d) in Figure \ref{behaviorcbc}).  For any $t\geq 0$, $z\in [0,\infty]$ and $x\in (0,\infty)$ 
$$\mathbb{E}_z(e^{-xZ_t})=\mathbb{E}_x(e^{-zU_t})\text{ with } \ddr U_t=\sqrt{cU_t}\ddr B_t+\lambda \ddr t, \ U_0=x,$$
the boundary $0$ of $(U_t,t\geq 0)$ being an entrance.
\end{example}
\begin{remark} Roughly speaking, the latter example can be seen as the continuous-state analogue of the number of fragments in a \textit{fast-fragmentation-coalescence} process as defined in \cite{kyprianou2017}. Note in particular that the same phase transition between $\infty$ reflecting and exit boundary occurs. However, contrary to the process in \cite{kyprianou2017}, the process $(Z_t,t\geq 0)$ has no stationary distribution over $(0,\infty)$.
\end{remark}
\noindent The conditions $\int^{\infty}\log(x)\pi(\ddr x)<\infty$ and $\lambda>0$ are not necessary for having respectively $\mathcal{E}=\infty$ and $\mathcal{E}<\infty$. 
In the following example $\lambda=0$, $\int^{\infty}\log(x)\pi(\ddr x)=\infty$ and a phase transition occurs between entrance and regular. 
\begin{example}\label{example3}
Consider $\lambda=0$, $\sigma\geq 0$, $\gamma\in \mathbb{R}$ and set $\pi(\ddr u)=\frac{\alpha}{u(\log u)^{\beta+1}}\mathbbm{1}_{\{u\geq 2\}}\ddr u$ for some $\alpha>0$, $\beta>0$. 
\begin{itemize}
\item[i)] If $\beta=1$ and $\frac{2\alpha}{c}\leq 1$ then $\mathcal{E}=\infty$ and $\infty$ is an entrance boundary (case (a) in Figure \ref{behaviorcbc}). 
\item[ii)] If $\beta=1$ and $\frac{2\alpha}{c}>1$ then $\mathcal{E}<\infty$ and $\infty$ is a regular reflecting boundary (case (b) in Figure \ref{behaviorcbc}). 
\item[iii)] If $\beta\in ]0,1[$, then $\mathcal{E}<\infty$ and $\infty$ is a regular reflecting boundary (case (b) in Figure \ref{behaviorcbc}).
\end{itemize}
\end{example}
\noindent The following proposition allows us to  study easily Example \ref{example3} and to construct more general explicit L\'evy measures for which $\infty$ is regular or entrance.
\begin{proposition}\label{example} 
Assume $\lambda=0$ and set $$\mathcal{E}':=\int_{0}^\theta \frac{1}{x}\exp\left(-\frac{2}{c}\int_{1}^{\infty}e^{-xv}\frac{\bar{\pi}(v)}{v}\ddr v\right)\ddr x$$
then  $\mathcal{E}<\infty$ if and only if $\mathcal{E}'<\infty$.
Moreover there exists a universal\footnote{in the sense that it does not depend on the L\'evy measure $\pi$} constant $\kappa>0$ and $C_1$, $C_2$ two non-negative constants such that 
$$C_1\int_0^\theta\frac{1}{x}\exp\left(-\frac{2\kappa}{c}\int_{1}^{1/x}\frac{\bar{\pi}(u)}{u}\ddr u\right)\ddr x\leq \mathcal{E}\leq C_2\int_0^\theta\frac{1}{x}\exp\left(-\frac{2}{c\kappa}\int_{1}^{1/x}\frac{\bar{\pi}(u)}{u}\ddr u\right)\ddr x$$
\end{proposition}
\section{Minimal process and time-change}\label{existence}
We will prove in this section Theorem 1 and show that the martingale problem \textbf{(MP)} for the minimal logistic CSBP is well-posed. As described in Definition 3.2 in \cite{MR2134113}, one way to construct a logistic CSBP is to start from an Ornstein-Uhlenbeck type process and to time change it in Lamperti's manner. The problem of explosion was not addressed in \cite{MR2134113} and lies at the heart of our study. We provide therefore some details. We start by recalling some known results about Ornstein-Uhlenbeck type processes. Consider $(Y_t, t\geq 0)$ a spectrally positive L\'evy process with Laplace exponent $-\Psi$, killed at $\infty$
at an independent exponential random variable $\mathbbm{e}_\lambda$ with parameter $\lambda:=-\Psi(0)\geq 0$. Set $(R_t, t\geq 0)$ the process satisfying 
\begin{equation} \label{OU} R_t=z+Y_t-\frac{c}{2}\int_{0}^{t}R_s\ddr s.
\end{equation}
There is a unique process $(R_t,t\geq 0)$ satisfying (\ref{OU}), see Sato \cite[Chapter 3, Section 17 page 104]{MR3185174}. By definition it is called an Ornstein-Uhlenbeck type process with L\'evy process $(Y_t,t\geq 0)$ and parameter $c/2$. Unkilled Ornstein-Uhlenbeck type processes have been deeply studied by Hadjiev \cite{MR889469}, Sato and Yamazato \cite{MR738769} and Shiga \cite{MR1061937}. From Lemma 17.1 in Sato \cite{MR3185174}, one has for any $\theta>0$ and any $s\geq 0$
\begin{equation}\label{LaplacetransformOU}
\mathbb{E}_z(e^{-\theta R_s})=\exp\left(-\theta e^{-\frac{c}{2}s}z+\int_{0}^{s}\Psi(e^{-\frac{c}{2}u}\theta) \ddr u\right).
\end{equation}
In particular, by letting $\theta$ to $0$, we see that the process $(R_t, t\geq 0)$ will never reach $\infty$ in finite time if it is not killed. In the unkilled case, it is shown in \cite{MR1061937} that if $(Y_t,t\geq 0)$ is not a subordinator then the process $(R_t, t\geq 0)$ is irreducible in $\mathbb{R}$. Namely, for any $a\in \mathbb{R}$, if $\sigma_a:=\inf\{t\geq 0, R_t\leq a\}$ then $\mathbb{P}_z(\sigma_a<\infty)>0$ for any $z>0$. On the other hand, if $-\Psi$ is the Laplace exponent of a subordinator with drift $\delta\geq 0$, then the process $(R_t, t\geq 0)$ is irreducible in $(\frac{2\delta}{c},\infty).$  Moreover, the process can be positive recurrent, null-recurrent or transient. \cite[Theorem 1.1]{MR1061937} states that $(R_t,t\geq 0)$ is recurrent (in a pointwise sense) if $\mathcal{E}=\infty$ and transient if $\mathcal{E}<\infty$, where we recall
\begin{equation*}
\mathcal{E}=\int_{0}^{\theta}\frac{1}{x}\exp \left({\frac{2}{c}\int_x^\theta \frac{\imf{\Psi}{u}}{u}\, \ddr u}\right)\ddr x. 
\end{equation*}If $\lambda=-\Psi(0)>0$, then one can easily see that $\mathcal{E}<\infty$, so that explosion by a single jump can be seen as a particular case of transience. 
We work in the sequel with the process $(R_t, t\geq 0)$ stopped on first entry into $(-\infty, 0)$. Define $\sigma_0:=\inf\{t\geq 0, R_t<0\}$, $\theta_t:=\int_{0}^{t\wedge \sigma_0}\frac{\ddr s}{R_s}$ and its right-inverse $t\mapsto C_t:=\inf\{u\geq 0; \theta_u>t\}\in [0,\infty]$. The Lamperti time-change of the stopped process $(R_t,t\geq 0)$ is the process $(\zmin_t,t\geq 0)$ defined by 
\begin{align*}
\zmin_t&=
\begin{cases}
R_{C_t} &  0\leq t<\theta_\infty\\
0& t\geq \theta_\infty \text{ and }\sigma_0<\infty\\
\infty& t\geq \theta_\infty   \text{ and } \sigma_0=\infty.
\end{cases}
\end{align*}
A first consequence of this definition is that the process $(\zmin_t, t\geq 0)$ hit its boundaries if and only if $\theta_\infty<\infty$. On the one hand, if $\sigma_0<\infty$ then $\zeta_\infty= \infty$ and $\zeta_0=\theta_\infty=\int_{0}^{\sigma_0}\frac{\ddr s}{R_s}$. On the other hand, if $\sigma_0=\infty$ then $\zeta_0=\infty$ and $\zeta_\infty=\theta_\infty=\int_{0}^{\infty}\frac{\ddr s}{R_s}$. Note that if $\lambda>0$, then $R_s=\infty$ for any $s\geq \mathbbm{e}_\lambda$ and the last integral is finite. 
\\

Recall (\ref{MP}) and the martingale problem (\textbf{MP}) defining the minimal logistic CSBP.
\begin{lemma}\label{existencemin}
The process $(\zmin_t,t\geq 0)$ is a minimal logistic continuous-state branching process. 
\end{lemma}
\begin{proof} 
Notice first that if the process $(\zmin_t,t\geq 0)$, as defined above, hits $0$ or $\infty$, then it is absorbed. Denote by $\mathscr{L}^Y$ the generator of the (possibly killed) L\'evy process $(Y_t,t\geq 0)$ and  $\mathscr{L}^R$ the generator of $(R_t,t\geq 0)$, which acts on $C^2_c([0,\infty))$ as follows
$$\mathscr{L}^Rf(z)=\mathscr{L}^Yf(z)-\frac{c}{2}zf'(z).$$
By It\^o's formula (or by applying \cite[Theorem 3.1]{{MR738769}}), one can see that the process $\left(f(R_t)-\int_{0}^{t}\mathscr{L}^{R}f(R_s)\ddr s, t\geq 0\right)$ is a local martingale. By definition of the time-change, for any $t\in [0,\theta_{\infty})$, $\int_{0}^{C_t}\frac{\ddr s}{R_s}=t$ and then $C_t=\int_{0}^{t}\zmin_s \ddr s.$ A (continuous) time-change of a local martingale remains a local martingale. Hence
$$t\in[0,\theta_\infty) \mapsto f(R_{C_t})-\int_{0}^{C_t}\mathscr{L}^Rf(R_s)\ddr s=f(\zmin_t)-\int_{0}^{t}\zmin_s \mathscr{L}^Rf(\zmin_s)\ddr s$$ 
is a local martingale. By definition, for any $z\geq 0$, $\mathscr{L}f(z)=z\mathscr{L}^{R}f(z)$ and  since $f$ has compact support then $\mathscr{L}f$ is bounded. Therefore the above local martingale has paths which are bounded on time-intervals $[0,t]$, so that it is a true martingale and $(\zmin_t,t\geq 0)$ solves \textbf{(MP)}. 
\end{proof}
\begin{lemma}\label{wellposed} There exists a unique minimal logistic CSBP. 
\end{lemma}
\begin{proof}
We have seen above how to construct a solution to the martingale problem. Only uniqueness has to be justified. Consider any solution $(Z_t,t<\zeta)$ to the martingale problem (\textbf{MP}). Set $C_t:=\int_{0}^{t}Z_s\ddr s$ for $t<\zeta$ and $C_t:=C_{\zeta}$ for all $t\geq \zeta$. Let $\theta_t:=\inf\{s\geq 0: C_s>t\}$ and $R_t:=Z_{\theta_t}$ for any time $t\in [0,C_\zeta)$. By definition, $R_{C_t}=Z_t$ and  thus $C_{\zeta}=\inf\{t\geq 0; R_t\notin (0,\infty)\}$. As in Lemma \ref{existencemin}, but in the opposite direction, one sees that the process $(R_t,t<C_{\zeta})$ solves the same martingale problem as an Ornstein-Uhlenbeck type process (with parameters $\tilde{\Psi}=\Psi+\lambda$ and $c/2$) stopped on first entry into $(-\infty, 0)$. The Ornstein-Uhlenbeck type process is uniquely defined in law (see \cite[Chapter 3, section 17]{MR3185174} where  existence and pathwise uniqueness of solution to (\ref{OU}) are established). Moreover, one can readily check that $C_t=\inf\{s\geq 0, \int_{0}^{s}\frac{\ddr u}{R_u}>t\}$, which entails that the law of $(Z_t,t\geq 0)$ is uniquely determined by the law of $(R_t,t\geq 0)$.
\end{proof}
\noindent We now gather some path properties of minimal logistic CSBPs obtained directly by time-change. 
\begin{lemma}\label{longtermminimal} Assume that $-\Psi$ is not the Laplace exponent of a subordinator. If $\mathcal{E}=\infty$, then for any $z>0$ $$\mathbb{P}_z(\zmin_t\underset{t\rightarrow \infty}{\longrightarrow} 0)=1,$$
If $\mathcal{E}<\infty$, then
$$\mathbb{P}_z(\zmin_t\underset{t\rightarrow \infty}{\longrightarrow} 0)=\frac{\int_{0}^{\infty} \frac{1}{x}e^{-zx-\int_{\theta}^{x}\frac{2\Psi(y)}{cy}\ddr y}\ddr x}{\int_{0}^{\infty}\frac{1}{x}e^{-\int_{\theta}^{x}\frac{2\Psi(y)}{cy}\ddr y}\ddr x}<1.$$
\end{lemma}
\begin{proof}
By construction, $\{\zmin_t \underset{t\rightarrow \infty}{\longrightarrow} 0\}=\{\sigma_0<\infty\}$ with $\sigma_0:=\inf\{t\geq 0, R_t\leq 0\}$. According to Patie \cite[Proposition 3]{MR2128632}, for any $z>a\geq 0$ and $\mu>0$
\begin{equation}\label{LaplacetransformhittingOU}\mathbb{E}_{z}[e^{-\mu \sigma_a}]=\frac{\int_{0}^{\infty} x^{\mu-1}e^{-zx-\int_{\theta}^{x}\frac{2\Psi(y)}{cy}\ddr y}\ddr x}{\int_{0}^{\infty}x^{\mu-1}e^{-ax-\int_{\theta}^{x}\frac{2\Psi(y)}{cy}\ddr y}\ddr x}.
\end{equation}
In order to make the paper selfcontained, a simple proof of (\ref{LaplacetransformhittingOU}) is provided in the Appendix (see Lemma \ref{hittingOU}). 
One can easily check that  if $\mathcal{E}=\infty$, then 
$$\mathbb{E}_z[e^{-\mu \sigma_0}]=\frac{\int_{0}^{\infty} x^{\mu-1}e^{-zx-\int_{\theta}^{x}\frac{2\Psi(y)}{cy}\ddr y}\ddr x}{\int_{0}^{\infty}x^{\mu-1}e^{-\int_{\theta}^{x}\frac{2\Psi(y)}{cy}\ddr y}\ddr x}\underset{\mu \rightarrow 0}{\longrightarrow} 1.$$
Therefore for any $z\in (0,\infty)$, $\mathbb{P}_z(\sigma_0<\infty)=1$. Now if $\mathcal{E}<\infty$, 
$$\mathbb{P}_z(\sigma_0<\infty)=\underset{\mu \rightarrow 0}{\lim} \mathbb{E}_z[e^{-\mu \sigma_0}]= \frac{\int_{0}^{\infty} x^{-1}e^{-zx-\int_{\theta}^{x}\frac{2\Psi(y)}{cy}\ddr y}\ddr x}{\int_{0}^{\infty}x^{-1}e^{-\int_{\theta}^{x}\frac{2\Psi(y)}{cy}\ddr y}\ddr x}<1.$$
\end{proof}
\begin{lemma}\label{totalprogeny} Assume that $-\Psi$ is not the Laplace exponent of a subordinator. Set $\zeta_a:=\inf\{t\geq 0; \zmin_t\leq a\}$. For any $z>a\geq 0$ and $\mu>0$, one has
$$\mathbb{E}_{z}[e^{-\mu \int_{0}^{\zeta_a}Z_s^{\text{min}}\ddr s}]=\frac{\int_{0}^{\infty} x^{\mu-1}e^{-zx-\int_{\theta}^{x}\frac{2\Psi(y)}{cy}\ddr y}\ddr x}{\int_{0}^{\infty}x^{\mu-1}e^{-ax-\int_{\theta}^{x}\frac{2\Psi(y)}{cy}\ddr y}\ddr x}.$$
\end{lemma}
\begin{proof}
By the time-change $\sigma_a=\int_{0}^{\zeta_a}Z_s^{\text{min}}\ddr s \text{ a.s}$ and the statement follows directly by (\ref{LaplacetransformhittingOU}).
\end{proof}
%
%
\noindent The next lemma establishes Theorem \ref{explosion}.
\begin{lemma}[Explosion] \label{explosionlemma} The minimal process explodes with positive probability if and only if $\mathcal{E}<\infty$. 
\end{lemma}
\begin{proof}
On the event $\{\sigma_0<\infty\}$,  $(\zmin_t, t\geq 0)$ converges towards $0$ almost-surely and thus does not explode.  We then focus on the event $\{\sigma_0=\infty\}$. If $\lambda>0$ then explosion is trivial. Assume now $\lambda=0$. By time-change, the process $(\zmin_t, t\geq 0)$ explodes if and only if $\int^{\infty}\frac{1}{R_s}\ddr s<\infty$. Assume first $\mathcal{E}=\infty$. Let $a>0$ and consider the successive excursions of $(R_t,t\geq 0)$ under $a$. Since the process is recurrent, there is an infinite number of such excursions. Let $b_0:=0$ and for any $n\geq 1$, $a_n:=\inf\{t>b_{n-1}, R_t\leq a\}$, $b_n:=\inf\{t>a_n, R_t>a\}$. We see that\begin{equation*}
\int_0^\infty \frac{\ddr s}{R_{s}}
\geq \sum_{n\geq 1} \int_{a_n}^{b_n} \frac{1}{R_s}\, \ddr s \geq \sum_{n\geq 1}\frac{b_n-a_n}{a}. 
\end{equation*}Since excursions are i.i.d and they have positive lengths, plainly $$\mathbb{E}_{z}[e^{-\int_0^\infty \frac{\ddr s}{R_{s}}},\sigma_0=\infty]
\leq \mathbb{E}_{z}[e^{-\frac{1}{a}\sum_{n\geq 1}(b_n-a_n)}]=0.$$
Therefore $\theta_\infty=\infty$. Hence, $\zmin_t<\infty$ for all $t>0$ and $\infty$ is inaccessible.  Now consider the case $\mathcal{E}<\infty$, the unstopped process $(R_t, t\geq 0)$ is transient, and the event $\{\sigma_0=\infty\}$ has positive probability. One has to check that the integral $\int_0^{\infty}\frac{1}{R_s}\ddr s$ is finite almost-surely on the event $\{\sigma_0=\infty\}$. 
Recall the Laplace transform (\ref{LaplacetransformOU}), one has $$\mathbb{E}_z(e^{-\theta R_s})=\exp\left(-\theta e^{-\frac{c}{2}s}z+\int_{0}^{s}\Psi(e^{-\frac{c}{2}u}\theta) \ddr u\right).$$
By the change of variables $v=e^{-\frac{c}{2}u}\theta$, we get $\int_{0}^{s}\Psi(e^{-\frac{c}{2}u}\theta)\ddr u=\int_{\theta e^{-\frac{2}{c}s}}^{\theta}\frac{2\Psi(v)}{cv}\ddr v$, therefore
$$\mathbb{E}_z(e^{-\theta R_s})=\exp\left(-\theta e^{-\frac{c}{2}s}z +\int_{\theta e^{-\frac{c}{2}s}}^{\theta}\frac{2\Psi(v)}{cv}\ddr v\right)$$
and the same change of variables $x=\theta e^{-\frac{c}{2}s}$ provides,
$$\int_{0}^{\infty}\mathbb{E}_z(e^{-\theta R_s})\ddr s=\frac{2}{c}\int_{0}^{\theta}\frac{1}{x}\exp\left(-xz+\int_{x}^{\theta}\frac{2\Psi(v)}{cv}\ddr v\right)\ddr x.$$
Since $\mathcal{E}<\infty$, the last integral is finite for any $\theta>0$. Let $b>0$. By Tonelli, one has 
\begin{align}\label{meanexplosion} \int_{0}^{\infty}\mathbb{E}_z\left(\frac{1-e^{-bR_s}}{R_s}, \sigma_0=\infty\right)\ddr s&=\int_{0}^{b}\int_{0}^{\infty}\mathbb{E}_z(e^{-\theta R_s},\sigma_0=\infty)\ddr s\ddr \theta \nonumber \\
&\leq \int_{0}^{b}\int_{0}^{\infty}\mathbb{E}_z(e^{-\theta R_s})\ddr s\ddr \theta\nonumber\\ 
&= \frac{2}{c}\int_{0}^{b}\ddr \theta\int_{0}^{\theta}\frac{1}{x}\exp\left(-xz+\int_{x}^{\theta}\frac{2\Psi(v)}{cv}\ddr v\right)\ddr x.\end{align}
The upper bound is finite since $\theta\in(0,b)\mapsto \int_{0}^{\theta}\frac{1}{x}e^{\int_{x}^{\theta}\frac{2\Psi(v)}{cv}\ddr v}\ddr x$ is bounded. Thus $$\mathbb{E}_z\left(\int_{0}^{\infty}\frac{1-e^{-bR_s}}{R_s}\ddr s,\sigma_0=\infty\right)<\infty.$$ 
We deduce then that on the event $\{\sigma_0=\infty\}$, $\int_{0}^{\infty}\frac{1-e^{-b R_s}}{R_s}\ddr s<\infty$ a.s. Since $\mathcal{E}<\infty$ then on $\{\sigma_0=\infty\}$, $R_s\underset{s\rightarrow \infty}{\longrightarrow} \infty$ a.s and $\frac{1-e^{-bR_s}}{R_s}\underset{s\rightarrow \infty}{\sim} \frac{1}{R_s}$ a.s. Therefore $$\mathbb{P}_z\left(\int_{0}^{\infty}\frac{\ddr s}{R_s} <\infty|\sigma_0=\infty\right)=1,$$ and the process $(\zmin_t,t\geq 0)$ explodes. 
\end{proof}
\begin{lemma}\label{irreducibleclass} When $\mathcal{E}<\infty$, explosion is almost-sure if and only if $-\Psi$ is the Laplace exponent of a subordinator.  
\end{lemma}
\begin{proof} 
We have seen in the proof of Lemma \ref{explosionlemma} that when $\mathcal{E}<\infty$, the following two events coincide $$\{\zeta_\infty=\infty\}=\{\sigma_0<\infty\}.$$
In the non-subordinator case, one has $\mathbb{P}_z(\sigma_0=\infty)<1$ since the unstopped process $(R_t,t\geq 0)$ is irreducible in $(-\infty,\infty]$. Assume that $-\Psi$ is the Laplace exponent of a subordinator with drift $\delta\geq 0$ (possibly killed at rate $\lambda$). We show that $\sigma_0=\infty$ a.s. Let $(Y_t,t\geq 0)$ denote the subordinator with Laplace exponent $-\Psi$. Since $Y_t\geq z+\delta t$ for all $t\geq 0$ $\mathbb{P}_z$-a.s, a comparison argument in (\ref{OU}) entails that $R_t\geq r_t$ for all $t\geq 0$, $\mathbb{P}_z$-a.s, with $(r_t,t\geq 0)$ the solution to $\ddr r_t=\delta \ddr t-\frac{c}{2}r_t\ddr t$ with $r_0=z$. We deduce that $R_t\geq e^{-\frac{c}{2}t}z+\frac{2\delta}{c}(1-e^{-\frac{c}{2}t})>0$ for all $t\geq 0$, $\mathbb{P}_z$-a.s. This entails $\mathbb{P}_z(\sigma_0=\infty)=1$ for any $z>0$. 
\end{proof}
\begin{remark} When $-\Psi$ is the Laplace exponent of a subordinator, the Ornstein-Uhlenbeck type process is irreducible in $(\frac{2\delta}{c},\infty)$. Namely, for any $z>\frac{2\delta}{c}$, the process starting from $z$ hit any value in $(\frac{2\delta}{c},\infty)$ with positive probability. By time-change, $(\zmin_t,0\leq t< \zeta_\infty)$ is therefore also irreducible in $(\frac{2\delta}{c},\infty)$.
\end{remark}

\begin{lemma} If $\lambda>0$, then the minimal process always explodes by a jump to $\infty$.
In other words, the two types of explosion cannot occur for a given process.
\end{lemma}
\begin{proof} 
Let $(\zmin_t,t\geq 0)$ be a minimal logistic CSBP with $\lambda>0$. By the time-change, $(R_t,t\geq 0):=(Z_{\theta_t},t\geq 0)$ is an Ornstein-Uhlenbeck type process killed at some exponential random variable $\mathbbm{e}_\lambda$ and stopped at its first entry in $(-\infty,0)$. Since for any $s<\mathbbm{e}_\lambda$, $R_{s}<\infty$, then  $Z_t=R_{C_t}<\infty$ for any $t<\theta_{\mathbbm{e}_\lambda}$. Therefore, the process cannot explode before $\theta_{\mathbbm{e}_\lambda}$ and on the event $\{\sigma_0=\infty\}$, explosion is made by a single jump which occurs at time $\theta_{\mathbbm{e}_\lambda}=\int_{0}^{\mathbbm{e}_\lambda}\frac{\ddr s}{R_s}$. 
\end{proof}

\begin{remark} We have seen that when the Ornstein-Uhlenbeck type process $(R_t,t\geq 0)$ is transient, the logistic CSBP explodes. Therefore, a logistic CSBP cannot grow indefinitely without exploding. This is a striking difference with CSBPs where indefinite growth with no explosion can occur when the L\'evy process $(Y_t,t\geq 0)$ drifts "slowly" towards $\infty$.
\end{remark}
\section{Infinity as an entrance boundary}\label{entrancesection}
The goal of this section is to show Theorem \ref{entranceboundarytheo}. The proof will follow from Lemmas \ref{entranceU} (part 1), \ref{dualnonexplosive} and \ref{Zextendedentrance}.
\subsection{Generalized Feller diffusions}\label{generalizedFellerdiffusion}
Recall $(U_t, t<\tau)$ with $\tau:=\inf\{t\geq 0; U_t\notin (0,\infty)\}$ solution to the sde 
\begin{equation}\label{generalizedFellerdiffusionequation}
U_t=x+\int_{0}^{t}\sqrt{cU_s}\ddr B_s-\int_{0}^{t}\Psi(U_s)\ddr s
\end{equation}
where $(B_t, t\geq 0)$ is a Brownian motion. The following observation is our starting point in the study of logistic continuous-state branching processes by duality.
\begin{lem}[Generator duality]\label{generatordualitylemma} 
For all $x\in [0,\infty[$ and $z\in [0,\infty[$, let $e_{x}(z)=e^{-xz}$, then  
\begin{equation}\label{generatorduality}
\mathscr{L}e_{x}(z)=\mathscr{A}e_z(x)
\end{equation}
with $$\mathscr{A}f(x)=\frac{c}{2}xf''(x)-\Psi(x)f'(x).$$
\end{lem}
\begin{proof}
One can readily check that for all $x$ and $z$ in $]0,\infty[$, $$\mathscr{L}e_{x}(z)=\Psi(x)ze_x(z)+\frac{c}{2}xz^{2}e_{x}(z)=-\Psi(x)\frac{\partial e_{z}(x)}{\partial x}+\frac{c}{2}x\frac{\partial^{2}e_{z}(x)}{\partial x^{2}}.$$
\end{proof}
\noindent Intuitively, integrating each side in (\ref{generatorduality}) should provide a duality at the level of semi-groups of the form:
\begin{equation*}
\imf{\mathbb{E}_z}{e^{-x \zmin_t}}=\imf{\mathbb{E}_x}{e^{-z U_t}}.\end{equation*}The study of the boundaries $0$ and $\infty$ of $(U_t,t\geq 0)$ would then provide the nature of boundaries $\infty$ and $0$ of $(\zmin_t,t\geq 0)$. However, there is not a unique semi-group associated to $\mathscr{A}$ as several boundary conditions are possible. Some precautions are then needed while showing the above duality. We gather in this section, the boundary conditions of the diffusion. The proofs of the following statements are rather technical and postponed in the Appendix.
\begin{lem}[Boundaries]\label{entranceU} The boundaries $0$ and $\infty$ of the diffusion with generator $\mathscr{A}$ are classified as follows:
\begin{itemize}
\item[1)] The boundary $0$ is an exit if
$\mathcal{E}=\int_{0}^\theta\frac{1}{x}\exp\left(-\frac{2}{c}\int_{x}^{\theta}\frac{\Psi(u)}{u}\ddr u\right)\ddr x=\infty$, regular if $\mathcal{E}<\infty$ and $0\leq \frac{2\lambda}{c}<1$, and  an entrance if  $\frac{2\lambda}{c}\geq 1$.
\item[2)] The boundary $\infty$ is inaccessible. It is an entrance boundary if $\int^{\infty}\frac{\ddr x}{\Psi(x)}<\infty$ and a natural one if $\int^{\infty}\frac{\ddr x}{\Psi(x)}=\infty$.
\end{itemize}
\end{lem}
\noindent When $\mathcal{E}<\infty$, $0$ is regular and there are several possibilities for extending the minimal diffusion after $\tau$. In the next lemma, we denote by $(U^0_{t},t\geq 0)$ the diffusion (\ref{generalizedFellerdiffusionequation}) with either $0$ regular absorbing or exit. 

\begin{lem}[Exit law from $(0,\infty)$]\label{stationarylawlemma} Assume $0\leq \frac{2\lambda}{c}<1$. 
\begin{itemize} 
\item[1)] Assume there exists $z\geq 0$, such that $\Psi(z)\geq 0$ (-$\Psi$ is not the Laplace exponent of a subordinator), then for all $x\geq 0$, $\mathbb{P}_x(\underset{t\rightarrow \infty}{\lim} U^0_{t}=0)=1$.
\item[2)] 
Assume $\Psi$ of the form 
\begin{equation*} \Psi(v)=-\lambda-\delta v-\int_{0}^{\infty}(1-e^{-vu})\pi(\ddr u)
\end{equation*}
with $\delta\geq 0$ and $\int_{0}^{\infty}(1\wedge u)\pi(\ddr u)<\infty$. 
Recall the condition 
$$\textbf{(A)} \qquad  \delta=0 \text{ and } \bar{\pi}(0)+\lambda\leq c/2$$  
\begin{itemize}
\item[i)] If $\textbf{(A)}$ is satisfied then for all $x\geq 0$, $\mathbb{P}_x(\underset{t\rightarrow \infty}{\lim} U^0_{t}=0)=1$.
\item[ii)] If $\textbf{(A)}$ is not satisfied then for all $x\geq 0$, 
$$\mathbb{P}_x(\underset{t\rightarrow \infty}{\lim} U^0_{t}=0)=1-\mathbb{P}_x(\underset{t\rightarrow \infty}{\lim} U^0_{t}=\infty)=
\frac{\int_{x}^{\infty}\exp\left(\int_{\theta}^{y}\frac{2\Psi(z)}{cz}\ddr z\right)\ddr y}{\int_{0}^{+\infty}\exp\left(\int_{\theta}^{y}\frac{2\Psi(z)}{cz}\ddr z\right)\ddr y}.$$
\end{itemize}
\end{itemize}
\end{lem}

\subsection{Duality and entrance law.}
In this section, we assume $\mathcal{E}=\infty$. Recall that it ensures the inaccessibility of $\infty$ for the process $(\zmin_t,t\geq 0)$ and that $0$ is an exit for the diffusion $(U_t,t\geq 0)$.

\begin{lem}[Duality lemma]\label{dualnonexplosive}
Assume $\mathcal{E}=\infty$. For all $z\in [0,\infty)$ and $x\in (0,\infty)$, 
the following duality holds $$\mathbb{E}_{z}[e^{-x\zmin_t}]=\mathbb{E}_{x}[e^{-zU_t}].$$
\end{lem}
\begin{proof}
Recall $(R_t,t\geq 0)$ the Ornstein-Uhlenbeck type process. For any $x>0$, by It\^o's formula, one sees that the process $$\left(e^{-xR_t}-\int_{0}^{t}\mathscr{L}^{R}e_{x}(R_s)\ddr s,t\geq 0\right)$$
is a local martingale. By time-changing it, we obtain that 
$$(M^{Z}_t,t\geq 0):=\left(e^{-x\zmin_t}-\int_{0}^{t}\mathscr{L}e_{x}(\zmin_s)\ddr s,t\geq 0\right),$$
is a local martingale. Since $x>0$, then $z\mapsto \mathscr{L}e_{x}(z)$ is bounded and $(M^Z_t,t\geq 0)$ is a martingale. Consider now $(U_t,t\geq 0)$ a $\Psi$-generalized Feller diffusion independent of $(\zmin_t,t\geq 0)$. By applying It\^ o's formula, we have that for any $z\geq 0$, and $\epsilon>0$;
$$(M^{U}_t,t\geq 0):=\left(e^{-zU_{t\wedge \tau_{\epsilon}}}-\int_{0}^{t\wedge \tau_{\epsilon}}\mathscr{A}e_{z}(U_s)\ddr s, t\geq 0\right)$$
is a martingale, with $\tau_\epsilon:=\inf\{t\geq 0, U_{t}\leq \epsilon\}$.  Recall the generator duality in Lemma \ref{generatordualitylemma}, $\mathscr{A}e_{z}(x)=\mathscr{L}e_{x}(z)$ and set $g(z,x):=\mathscr{L}e_{x}(z)$. We apply Ethier-Kurtz's duality result \cite[Corollary 4.15 page 196]{EthierKurtz} (with $\alpha=\beta=0$, $\tau=\infty$, $\sigma=\tau_\epsilon$). Provided that their condition (4.50) holds, we obtain, for $x\geq \epsilon$
\begin{align*}
\mathbb{E}_{z}[e^{-x\zmin_t}]-\mathbb{E}_x[e^{-zU_{t\wedge \tau_{\epsilon}}}]&=\int_{0}^{t}\mathbb{E}\left[\mathbbm{1}_{t-s>\tau_{\epsilon}}g(\zmin_{s},U_{(t-s)\wedge \tau_{\epsilon}})\right]\ddr s\\
&=\mathbb{E}\left[\int_{0}^{t-\tau_{\epsilon}\wedge t}\mathscr{L}e_\epsilon(\zmin_{s})\ddr s\right].
\end{align*}
From the last equality, it comes
\begin{align*}
\mathbb{E}_x[e^{-zU_{t\wedge \tau_{\epsilon}}}]-\mathbb{E}_{z}[e^{-x\zmin_t}]
&=-\mathbb{E}\left[\int_{0}^{t-\tau_{\epsilon}\wedge t}\mathscr{L}e_\epsilon(\zmin_{s})\ddr s\right]=e^{-\epsilon z}-\mathbb{E}_{z}[e^{-\epsilon \zmin_{t-t\wedge \tau_{\epsilon}}}],
\end{align*}
where we have obtained the last equality using the martingale $(M^Z_t,t\geq 0)$ conditionally given $\tau_\epsilon$ since $\tau_{\epsilon}$ is independent of $(\zmin_t,t\geq 0)$. By letting $\epsilon$ to $0$, $\tau_{\epsilon}\underset{\epsilon \rightarrow 0}{\longrightarrow} \tau_0$ a.s and the last equality provides
$$\mathbb{E}_x[e^{-zU_{t\wedge \tau_{\epsilon}}}]-\mathbb{E}_{z}[e^{-x\zmin_t}]\underset{\epsilon \rightarrow 0}{\longrightarrow} 1-\mathbb{P}_{z}(\zmin_{t-t\wedge \tau_0}<\infty).$$
We know that under the condition $\mathcal{E}=\infty$ the process $(\zmin_t,t\geq 0)$ does not explode. Therefore the limit above is $0$ and 
$\mathbb{E}_{z}[e^{-x\zmin_t}]=\mathbb{E}_x[e^{-zU_{t\wedge \tau_{0}}}]$ for all $x>0$ and $z\in [0,\infty)$. On the other hand, under the condition $\mathcal{E}=\infty$, $0$ is an exit of the diffusion and thus
$$\mathbb{E}_{z}[e^{-x\zmin_t}]=\mathbb{E}_x[e^{-zU_{t}}].$$
It remains to verify the technical condition (4.50) in Ethier-Kurtz \cite{EthierKurtz} page 192. Namely, for any  $T>0$ and $\epsilon>0$ fixed, we need to show that  the random variables $\sup_{s,t\leq T} \exp\left(-U_{s\wedge \tau_{\epsilon}}\zmin_{t}\right)$ and $$\sup_{s,t\leq T} |g(\zmin_{t},U_{s\wedge \tau_{\epsilon}})|,\text{ where } g(z,u)=\Psi(u)ze^{-uz}+\frac{c}{2}uz^2e^{-uz}$$
are integrable. Since $\sup_{s,t\leq T} \exp\left(-U_{s\wedge \tau_{\epsilon}}\zmin_{t}\right)$ is clearly bounded by $1$, we only need to focus on $\sup_{s,t\leq T} |g(\zmin_{t},U_{s\wedge \tau_{\epsilon}})|$.
For any mechanism $\Psi$, if $u\geq \epsilon>0$, then $|\Psi(u)|\leq b_\epsilon u^2$ for some $b_\epsilon>0$. Recall $x\geq \epsilon>0$. For all $s\geq 0$, $U_{s\wedge \tau_{\epsilon}}\geq \epsilon$ a.s. under $\mathbb{P}_x$, therefore
\begin{align*}
|g(\zmin_{t},U_{s\wedge \tau_{\epsilon}})|&=\lvert \Psi(U_{s\wedge \tau_{\epsilon}})\zmin_te^{-U_{s\wedge \tau_{\epsilon}}\zmin_t}+\frac{c}{2}U_{s\wedge \tau_\epsilon}(\zmin_t)^{2}e^{-U_{t\wedge \tau_\epsilon}\zmin_t}\lvert\\
&\leq b_\epsilon U_{s\wedge \tau_{\epsilon}}^2\zmin_te^{-U_{s\wedge \tau_{\epsilon}} \zmin_t}+\frac{c}{2}U_{s\wedge \tau_\epsilon}(\zmin_t)^{2}e^{-U_{s\wedge \tau_\epsilon}\zmin_t}\\
&\leq b_\epsilon U_{s\wedge \tau_{\epsilon}}^2\zmin_te^{-\epsilon \zmin_t}+\frac{c}{U_{s\wedge \tau_\epsilon}}\\
&\leq \frac{b_\epsilon}{\epsilon}U^2_{s\wedge \tau_\epsilon}+\frac{c}{\epsilon}
\end{align*}
where in the second inequality we have used $uz^2e^{-uz}\leq \frac{2}{u}$ and in the third one, $ze^{-\epsilon z}\leq \frac{1}{\epsilon}$. We now argue by comparison in order to show that 
$\sup_{s\leq T}U_{s\wedge \tau_\epsilon}^2$ is integrable. When $u\geq \epsilon$, we have $\Psi(u)\geq \frac{\Psi(\epsilon)}{\epsilon}u\geq -\gamma_\epsilon u$ for some $\gamma_{\epsilon}>0$. Recall that $\Psi$ is locally Lipschitz on $(0,\infty)$. Applying the results of \cite[Section 3, Chapter IX]{MR1725357}, one can then construct, with the same Brownian motion $(B_t,t\geq 0)$, the process $(U_{t}, t\geq 0)$ as a strong solution to (\ref{generalizedFellerdiffusionequation}) with $0$ exit and the process $(V_t,t\geq 0)$ as a strong solution to
$$\ddr V_t=\sqrt{cV_t}\ddr B_t+\gamma_\epsilon V_t\ddr t, \quad V_0=x.$$
Both processes are adapted to the natural filtration of $(B_t,t\geq 0)$. Applying the comparison theorem  \cite[Theorem IX.3.7]{MR1725357} up to the stopping time $\tau_\epsilon$, one has that almost-surely for any $0\leq s \leq \tau_\epsilon$,  $U_{s}\leq V_{s}$. Note that $(V_t,t\geq 0)$ is a supercritical Feller diffusion with branching mechanism $\Phi(u)=\frac{c}{2}u^{2}+\gamma_\epsilon u$. It is easily checked that for any $t\geq 0$, $V_t$ has a second moment. Moreover, the process $(V_s,s\geq 0)$ is a submartingale and by Doob's inequality
$$\mathbb{E}_{x}\left[\sup_{s\leq T}V^2_s\right]\leq 4\mathbb{E}_{x}[V_T^{2}]<\infty.$$
Since for any $\epsilon>0$, $\underset{s\leq T\wedge \tau_\epsilon}{\sup} U^2_{s}\leq\underset{s\leq T}{\sup}\ V^2_{s}$, the proof is complete.
\end{proof}
Let $\paren{\pmin_t,t\geq 0}$ be the semigroup of $(\zmin_t,t\geq 0)$. Lemma \ref{explosionlemma} ensures that when $\mathcal{E}=\infty$, $\infty$ is inaccessible. To see that $\infty$ is an entrance boundary,  we show in the following lemmas how to define a Feller semigroup coinciding with $\paren{\pmin_t,t\geq 0}$ over $[0,\infty)$, with an entrance law from $\infty$.  
\begin{lem}
For any $t>0$, $x\mapsto \imf{\mathbb{P}_x}{U_t=0}$ is the Laplace transform of a certain probability measure $\eta_t$ over $[0,\infty)$. 
Moreover $\eta_t\to \eta_0:=\delta_\infty$ weakly as $t\to 0$. 
\label{entranceLawAtInfinityLemma}
\end{lem}

\begin{proof}
By taking limits as $z\to\infty$ in the duality formula in Lemma \ref{dualnonexplosive}, one has: 
\begin{equation*}
\lim_{z\to\infty}\imf{\mathbb{E}_z}{e^{-x\zmin_t}}=\lim_{z\to\infty}\imf{\mathbb{E}_x}{e^{-zU_t}}=\imf{\mathbb{P}_x}{U_t=0}.
\end{equation*}Since $0$ is an exit thanks to the assumption $\mathcal{E}=\infty$, $\imf{\mathbb{P}_x}{U_t=0}=\imf{\mathbb{P}_x}{\tau_0\leq t}>0$. By L\'evy 's continuity theorem, $x\mapsto \imf{\mathbb{P}_x}{\tau_0\leq t}$ is the Laplace transform of a certain finite measure $\eta_t$ which is the weak limit of the law of $\zmin_t$ under $\mathbb{P}_z$ as $z\to\infty$. 
Moreover, $\underset{x\rightarrow 0}{\lim}\imf{\mathbb{P}_x}{\tau_0\leq t}=\imf{\mathbb{P}_{0+}}{\tau_0\leq t}=1$ 
and $\eta_t$ is a probability measure over $[0,\infty)$. 
By continuity of the paths of $(U_t,t\geq 0)$, if $x>0$, then $\underset{t\rightarrow 0}{\lim} \mathbb{P}_{x}(U_t=0)=0$, and if $x=0$ then $\underset{t\rightarrow 0}{\lim} \mathbb{P}_{x}(U_t=0)=1$. This entails that  $\eta_t\to\delta_\infty$ weakly as $t\to 0$.
\end{proof}

From now on, we will work with the following definition of $e_x$ over $[0,\infty]$. For any $x>0$, $e_{x}(z)=e^{-xz}$ for all $z\in [0,\infty]$ and $e_{0}(z)=1$ for all $z\in [0,\infty]$. Note that $e_{0+}(z):= \underset{x\rightarrow 0 \atop x>0}\lim e_x(z)=\mathbbm{1}_{\{z<\infty\}}.$
\begin{lem}\label{Zextendedentrance}
For any function $f\in C_b([0,\infty])$ and any $t\geq 0$, set $P_tf(z):=\pmin_tf(z)$ for any $z\in [0,\infty)$ and $P_tf(\infty):=\int_{0}^{\infty}f(u)\eta_{t}(\ddr u).$ 
This defines a Feller semigroup $\paren{P_t, t\geq 0}$ over $[0,\infty]$. 
Furthermore, if $(Z_t,t\geq 0)$ is a c\`adl\`ag Markov process with semigroup $(P_t,t\geq 0)$, and $T:=\inf\{t>0; Z_t<\infty\}$, then $\imf{\mathbb{P}_\infty}{T=0}=1$. 
\end{lem}
\begin{proof}
The subalgebra generated by the maps $\{e_x(\cdot),x\geq 0\}$ is separating $[0,\infty]$ (recall that $e_0(z)=1$ for any $z\in [0,\infty]$) and therefore by the Stone-Weierstrass theorem, is dense in $C_b([0,\infty])$ for the supremum norm. By duality, for any $x\geq 0$, 
$P_t e_x(z)=\mathbb{E}_{x}[e^{-zU_t}]$ when $z\in [0,\infty)$. The map $z\mapsto P_t e_x(z)$ is therefore continuous on $[0,\infty[$. The continuity at $z=\infty$ holds since by definition $P_t e_x(\infty)=\underset{z\rightarrow \infty}{\lim} P_t e_x(z)$.
By Stone-Weierstrass, $z\mapsto P_t f(z)$ is continuous on $[0,\infty]$ with any $f\in C_b([0,\infty])$. Hence $P_t C_b([0,\infty])\subset C_b([0,\infty])$. 
We show now that $(P_t,t\geq 0)$ is a semigroup. Since it coincides with the semigroup $(\pmin_t,t\geq 0)$ on $[0,\infty[$, then for any $s,t\geq 0$, any function $f\in C_b([0,\infty])$ and any $z\in [0,\infty)$ $\imf{P_{t+s}f}{z}=\imf{P_{t}P_sf}{z}$. For $z=\infty$, we have
\begin{equation*}
P_{t+s}f(\infty)=
\lim_{z\to\infty} \imf{P_{t+s}f}{z}
=\lim_{z\to\infty} \imf{P_{t}P_sf}{z}
=\int  \imf{P_sf}{y}\, \imf{\eta_t }{\ddr y}.
\end{equation*}The last equality above holds since $P_sf \in C_b([0,\infty])$. This provides $P_{t+s}f(\infty)=\imf{P_{t}P_sf}{\infty}$. It remains to justify the continuity of $(P_t, t\geq 0)$ at $0$. That is to say $P_tf(z)\underset{t\rightarrow 0}{\longrightarrow} f(z)$ for any $z\in [0,\infty]$ and any $f\in C_b([0,\infty])$. Since $(U_t,t\geq 0)$ is a diffusion (with continuous paths and with infinite life-time) then $t\mapsto P_t e_x(z)=\imf{\mathbb{E}_x}{e^{-z U_t}}$ is continuous, in particular continuous at $0$. Hence, $t\mapsto P_tf(z)$ is continuous at zero for any $z\in [0,\infty)$. For $z=\infty$, since $\eta_t \underset{t\rightarrow 0}{\longrightarrow} \delta_\infty$ weakly, then $P_tf(\infty)\underset{t\rightarrow 0}{\longrightarrow} f(\infty)$. 
This entails the Feller property of $(P_t,t\geq 0)$ (see e.g. \cite[Section 2, Chapter III]{MR1725357}), which ensures the  existence of a Markov process $(Z_t,t\geq 0)$ with semigroup $(P_t,t\geq 0)$ and c\`adl\`ag paths. We show now that $\infty$ is instantaneous. Since for every $t>0$, $\eta_t$ is a probability over $\mathbb{R}_+$, then
$\mathbb{P}_\infty(T<t)=\mathbb{P}_\infty(Z_t<\infty)=\eta_t(\mathbb{R}_+)=1.$
Letting $t$ to $0$ provides $\mathbb{P}_\infty(T=0)=1$.  
\end{proof}
%
%
We give in the next lemma, an alternative proof for the property of entrance at $\infty$, based on arguments that are not involving duality.

\begin{lem} Define $\zeta_a:=\inf\{t\geq 0; \zmin_t\leq a\}$ for any $a\geq 0$. 
For any large enough positive $a$, one has $\underset{z\geq a}{\sup }\ \mathbb{E}_z(\zeta_a)\leq \frac{4}{ca}$.
\end{lem}
\begin{proof}
Recall $\mathscr{G}$ the generator of a CSBP with mechanism $\Psi$. Let $h(z)=\frac{1}{z}$. One has 
\begin{align*}
\mathscr{G}h(z)&=\frac{\sigma^2}{z^2}+\frac{\gamma}{z}+\int_{0}^{\infty} z\left( \frac{1}{z+h}-\frac{1}{z}+\mathbbm{1}_{\{h\leq 1\}}h\frac{1}{z^2}\right)\pi(\ddr h)\\
&=\frac{\sigma^2}{z^2}+\frac{\gamma}{z}-\int_{1}^{\infty}\frac{h}{z+h}\pi(\ddr h)+\int_{0}^{1}h\left(\frac{1}{z}-\frac{1}{z+h}\right)\pi(\ddr h).
\end{align*}
By Lebesgue's theorem, $\mathscr{G}h(z)\underset{z\rightarrow \infty}{\longrightarrow} 0.$ Since $\mathscr{L}h(z)=\mathscr{G}h(z)+\frac{c}{2}$, then there exists $a>0$ such that for all $z\geq a$, $\mathscr{L}h(z)\geq \frac{c}{4}$. Since by assumption, $\mathcal{E}=\infty$, the process $(\zmin_t,t\geq 0)$ does not explode and there exists a localizing sequence of stopping times $(T_m, m\geq 1)$ such that $T_m \underset{m\rightarrow \infty}{\longrightarrow} \infty$ almost-surely and $(M_{t\wedge T_m}, t\geq 0)$ is a bounded martingale, where
\[M_t=h(\zmin_t)-\int_0^t\mathscr{L}h(\zmin_s)\ddr s.\]
By the optional stopping theorem, $\mathbb{E}_z[M_{\zeta_{a}\wedge T_m}]=h(z)$ and we obtain, letting $m$ to $\infty$
$$\mathbb{E}_z[h(\zmin_{\zeta_{a}})]-h(z)=\frac{1}{a}-\frac{1}{z}=\mathbb{E}_z\left[\int_{0}^{\zeta_{a}}\mathscr{L}h(\zmin_s)\ddr s\right]\geq \frac{c}{4}\mathbb{E}_z(\zeta_a).$$
We conclude that  $\mathbb{E}_z(\zeta_a)\leq \frac{4}{c}\left(\frac{1}{a}-\frac{1}{z}\right)$ for any $z\geq a$. The entrance property can be deduced by following the proof of Kallenberg \cite[Theorem 23.13]{Kallenberg}.
\end{proof}
\subsection{Longterm behavior and stationarity}
We now show Corollary \ref{zeroboundary}, Corollary \ref{StationarityZ} and Theorem \ref{longtermtheo} in the case $\mathcal{E}=\infty$. 
\begin{lem}[Corollary \ref{zeroboundary}: accessibility of $0$] \label{0access} Let $\zeta_0:=\inf\{t>0; Z_t=0\}$. If $\int^{\infty}\frac{\ddr u}{\Psi(u)}<\infty$ then for any $z\geq 0$, $\mathbb{P}_z(\zeta_0<\infty)>0$. If $\int^{\infty}\frac{\ddr u}{\Psi(u)}=\infty$ then for any $z>0$, $\mathbb{P}_z(\zeta_0=\infty)=1$.
\end{lem}
\begin{proof}
For all $z\in \mathbb{R}_+$, $$\mathbb{P}_z(\zeta_0\leq t)=\underset{x\rightarrow \infty}{\lim}\mathbb{E}_{z}[e^{-xZ_t}]=\underset{x\rightarrow \infty}{\lim}\mathbb{E}_{x}[e^{-zU_t}]=
\mathbb{E}_{\infty}[e^{-zU_t}].$$
Therefore $\mathbb{P}_z(\zeta_0\leq t)>0$ if and only if $\mathbb{E}_{\infty}[e^{-zU_t}]>0$. By Lemma \ref{entranceU}, $\infty$ is an entrance boundary of the diffusion $(U_t, t\geq 0)$ if and only if $\int^{\infty}\frac{\ddr u}{\Psi(u)}<\infty$. 
\end{proof}
\begin{lem}[Corollary \ref{StationarityZ}: stationarity]\label{stationarityandalmostsureabsorptionEinfini} If the assumption $\textbf{(A)}$ (with $\lambda=0$) is not satisfied then for all $x\geq 0$
$$\mathbb{E}_z\left(e^{-xZ_t}\right)\underset{t\rightarrow \infty}{\longrightarrow} L(x):= \frac{\int_{x}^{\infty}\exp\left(\int_{\theta}^{y}\frac{2\Psi(z)}{cz}\ddr z\right)\ddr y}{\int_{0}^{\infty}\exp\left(\int_{\theta}^{y}\frac{2\Psi(z)}{cz}\ddr z\right)\ddr y}.$$
Moreover $L$ is the Laplace transform of a probability measure carried over $(\frac{2\delta}{c},\infty)$, where $\delta=-\underset{u\rightarrow \infty}{\lim}\frac{\Psi(u)}{u}$.  If the assumption $\textbf{(A)}$ is satisfied or $-\Psi$ is not the Laplace exponent of a subordinator then for all $x\geq 0$,
$$\mathbb{E}_z\left(e^{-xZ_t}\right)\underset{t\rightarrow \infty}{\longrightarrow} 1.$$ 
\end{lem}
\begin{proof} According to Lemma \ref{stationarylawlemma}, $(U_t,t\geq 0)$ exits $(0,\infty)$ either by $0$ or by $\infty$. Thus,
for any $z\in (0,\infty]$ and $x\geq 0$, $$\mathbb{E}_z(e^{-xZ_t})=\mathbb{E}_x(e^{-zU_t}\mathbbm{1}_{\{\underset{t\rightarrow \infty}\lim U_t=0\}}+e^{-zU_t}\mathbbm{1}_{\{\underset{t\rightarrow \infty}\lim U_t=\infty\}})\underset{t\rightarrow \infty}{\longrightarrow} \mathbb{P}_x(\underset{t\rightarrow \infty}{\lim}U_t=0).$$
A direct application of Lemma \ref{stationarylawlemma} provides the two first convergences. The support of the stationary measure is $(\frac{2\delta}{c},\infty)$ since $(\zmin_t,t\geq 0)$ is irreducible in $(\frac{2\delta}{c},\infty)$, see the Remark below Lemma \ref{irreducibleclass}. 
\end{proof}The next Lemma establishes part 1) of Theorem \ref{longtermtheo} under the additional condition $\mathcal{E}=\infty$. Later we get the same part proved under $\mathcal{E}<\infty$.
\begin{lem}
Assume $\Psi(z)\geq 0$ for some $z>0$  then
\begin{itemize}
\item[1)] If $\int^{\infty}\frac{\ddr u}{\Psi(u)}=\infty$, then $Z_t>0$ for any $t\geq 0$ a.s. and $Z_t\underset{t\rightarrow \infty}{\longrightarrow} 0$ a.s.
\item[2)] If $\int^{\infty}\frac{\ddr u}{\Psi(u)}<\infty$, then $(Z_t,t\geq 0)$ get absorbed at $0$ in finite time almost-surely.
\end{itemize}
\end{lem}
\begin{proof}
Note first that Lemma \ref{longtermminimal} ensures that $Z_t\underset{t\rightarrow \infty}{\longrightarrow} 0$ a.s. If $\int^{\infty}\frac{\ddr z}{\Psi(z)}=\infty$, we have seen in Lemma \ref{0access} that $0$ is inaccessible. Assume now $\int^{\infty}\frac{\ddr z}{\Psi(z)}<\infty$, then by Lemma \ref{stationarylawlemma}-1) $\mathbb{P}_z(\zeta_0\leq t)=\mathbb{E}_\infty(e^{-zU_t})\underset{t\rightarrow \infty}{\longrightarrow} 1,$
thus $\mathbb{P}_z(\zeta_0<\infty)=1$.
\end{proof}
\section{Infinity as regular reflecting or exit boundary}
\label{regularsection}

In this subsection, we assume that $\mathcal{E}$ is finite and will prove  Theorems \ref{regularboundarytheo} and \ref{exitboundarytheo}. Recall that if $0\leq 2\lambda/c<1$ then $0$ is a regular boundary of the $\Psi$-generalized Feller diffusion and if $2\lambda/c\geq 1$, then $0$ is an entrance boundary. \\

\noindent Recall $\Psi_{k}$ the branching mechanism associated to the triplet $(\sigma, \gamma, \pi_k)$ with $\pi_k(\ddr u)=\pi_{|]0,k[}(\ddr u)+(\bar{\pi}(k)+\lambda)\delta_k $, that is
$$\Psi_{k}(x)=\frac{\sigma^2}{2}x^2-\gamma x + \int_{]0,k[}\left(e^{-xu}-1+xu1_{u\in ]0,1[}\right)\pi(\ddr u)+ (e^{-xk}-1)(\bar{\pi}(k)+\lambda).$$
Note that for all $k\geq 0$, $|\Psi_{k}'(0+)|<\infty$ and for all $x>0$, $\Psi_{k}(x)\underset{k\rightarrow \infty}{\longrightarrow} \Psi(x)$.  
\begin{lemma}\label{approxU0} 
For any $k\geq 0$, denote $(U^{(k)}_t, t\geq 0)$ the unique solution to $$\ddr U^{(k)}_t=\sqrt{cU^{(k)}_t}\ddr B_t-\Psi_{k}(U^{(k)}_t)\ddr t.$$
There exists a probability space on which, with probability $1$: 
\[U^{(k)}_t\leq U^{(k+1)}_t, \text{ for all } k\geq 1 \text{ and } t\geq 0.\]
If $0\leq \frac{2\lambda}{c}<1$, then as $k$ goes to $\infty$, the sequence of processes $(U^{(k)}_t, t\geq 0)$ converges pointwise almost-surely towards a process $(U^0_{t}, t\geq 0)$ which is the solution to (\ref{generalizedFellerdiffusionequation}) absorbed at $0$. If $\frac{2\lambda}{c}\geq 1$, then as $k$ goes to $\infty$, the sequence of processes $(U^{(k)}_t, t\geq 0)$ converges almost-surely towards $(U_{t}, t\geq 0)$ the unique solution to (\ref{generalizedFellerdiffusionequation}).
\end{lemma}

\begin{proof}
One can plainly check that for any $k\geq 0$ and $x\geq 0$, $\Psi_{k}(x)\geq \Psi_{k+1}(x)$. Since $|\Psi_{k}'(0+)|<\infty$ for any $k\geq 0$, the branching mechanisms $\Psi_{k}$ are locally Lipschitz on $[0,\infty)$ and therefore by applying the comparison theorem for SDEs (\cite[Theorem IX.3.7]{MR1725357}), one has on some probability space $$\mathbb{P}_x\left(U^{(k+1)}_t\geq U^{(k)}_t \text{ for all } t\geq 0 \text{ and all } k\geq 0\right)=1.$$
The existence of the limiting process $(U^{(\infty)}_t,t\geq 0)$ (in a pointwise sense) is ensured by monotonicity. Set $\tau^{k}:=\inf\{t\geq 0, U^{(k)}_t\notin (0,\infty)\}$ and
$\tau^{\infty}:=\inf\{t\geq 0, U^{(\infty)}_t\notin (0,\infty)\}$.  
Let $\mathscr{A}^{(k)}$ be the generator of $(U^{(k)}_t,t\geq 0)$. The diffusive part in $\mathscr{A}^{(k)}$ and $\mathscr{A}$ are the same, therefore for any $g\in C^{2}_c((0,\infty))$:
\begin{align*}
||\mathscr{A}^{(k)}g-\mathscr{A}g||_\infty&=\sup_{x\in [0,\infty[}\lvert(\Psi(x)-\Psi_{k}(x))g'(x)\lvert\\
&=\sup_{x\in [0,\infty[}\left\lvert\left(-\lambda+\int_{]k,\infty]}(e^{-xu}-1)\pi(\ddr u)+(1-e^{-xk})(\bar{\pi}(k)+\lambda)\right)g'(x)\right\lvert\\
&\leq 2\sup_{x\in [0,\infty[}\lvert e^{-xk}\bar{\pi}(k)g'(x)\lvert+\lambda \sup_{x\in [0,\infty[}\lvert e^{-xk}g'(x)\lvert.
\end{align*} 
For all $k\geq 1$ $\bar{\pi}(k)\leq \bar{\pi}(1)$  and for any $x>0$, $\underset{k\rightarrow \infty}{\lim} e^{-xk}=0$, since $g'$ has a compact support in $(0,\infty)$, then $||\mathscr{A}^{(k)}g-\mathscr{A}g||_\infty \underset{k\rightarrow \infty}{\longrightarrow} 0$. Hence, for large enough $k\geq 1$, $||\mathscr{A}^{(k)}g||_\infty\leq 1+||\mathscr{A}g||_\infty$ and  $\mathscr{A}^{(k)}g(U_s^{(k)})\underset{k\rightarrow \infty}{\longrightarrow} \mathscr{A}g(U_s^{(\infty)})$ a.s. for any $s\geq 0$. One can deduce that for any function $g\in \mathcal{C}^2_c((0,\infty))$, the process
$$t\in [0,\tau^{\infty})\mapsto g(U_t^{(\infty)})-\int^{t}_0\mathscr{A}g(U_s^{(\infty)})\ddr s$$
is a martingale (see for instance \cite[Lemma 5.1 page 196]{EthierKurtz}). The process $(U_t^{(\infty)}\!,t<\tau^{\infty})$ solves the same martingale problem as $(U_t,t<\tau)$. The latter problem being well-posed (see for instance \cite[Section 6.1, Theorem 1.6]{MR1398879}), $(U_t^{(\infty)}\!,t<\tau^{\infty})$ and $(U_t,t<\tau)$ have the same law. It remains only to identify the behavior of the process $(U_t^{(\infty)},t\geq 0)$ after its first explosion time $\tau^{\infty}$. Recall that by Lemma \ref{entranceU}, $\infty$ is inaccessible, so that $\tau^{\infty}=\inf\{t\geq 0, U^{(\infty)}_t=0\}$ a.s. If $\frac{2\lambda}{c}\geq 1$ then $0$ is an entrance, $\tau^{\infty}=\infty$ a.s. and $(U_t^{(\infty)},t\geq 0)$ has the same law as $(U_t,t\geq 0)$. If $0\leq \frac{2\lambda}{c}<1$, then $0$ is regular or exit and $\tau^{\infty}<\infty$ with positive probability.  Let $t\geq 0$. On the event $\{\tau^{\infty}<\infty\}$, by pointwise almost-sure convergence, $U_{t+\tau^{\infty}}^{(\infty)}=\underset{k\rightarrow \infty}{\lim} U^{(k)}_{t+\tau^{\infty}}$. By monotonicity $\tau^{\infty}\geq \tau^{k}$ a.s. for all $k\geq 1$, moreover $0$ is an exit for $(U^{(k)}_t,t\geq 0)$ therefore $U^{(k)}_{t+\tau^{\infty}}=U^{(k)}_{t+\tau^{k}}=0$ and $U_{t+\tau^{\infty}}^{(\infty)}=0$ for any $t\geq 0$. We conclude that $(U_t^{(\infty)},t\geq 0)$ is the diffusion whose generator is $\mathscr{A}$ with $0$ absorbing.
\end{proof}
The next lemmas  will provide proofs of Theorem \ref{regularboundarytheo} and Theorem \ref{exitboundarytheo}. For any $k\geq 1$, denote by $(Z_t^{(k)}, t\geq 0)$ the logistic CSBP defined on $[0,\infty]$ with mechanisms $\Psi_{k}$. According to Lemma \ref{Zextendedentrance}, the processes $(Z_t^{(k)},t\geq 0)$ are Feller. In the sequel we work with their c\`adl\`ag versions.
\begin{lemma}\label{weakconv} Assume $\mathcal{E}<\infty$ and $ 0\leq \frac{2\lambda}{c}<1$, the sequence $((Z^{(k)}_t)_{t\geq 0}, k\geq 1)$ converges weakly towards a c\`adl\`ag Feller process $(Z_t, t\geq 0)$ valued in $[0,\infty]$ such that for all $z\in [0,\infty]$, all $t\geq 0$, and all $x\in [0,\infty[$ $$\mathbb{E}_{z}[e^{-xZ_t}]=\mathbb{E}_{x}[e^{-zU^0_{t}}]$$
where $(U^0_t,t\geq 0)$ is the $\Psi$-generalized Feller diffusion satisfying (\ref{generalizedFellerdiffusionequation}) with $0$ regular absorbing.
\end{lemma}
\begin{proof}
Denote by $(P_t^{(k)},t\geq 0)$ the semi-group of $(Z_t^{(k)},t\geq 0)$ and $(p^{(k)}_t(z,\cdot), z\in [0,\infty],t\geq 0)$ its transition kernel. Let $t\geq 0$ and $z\in [0,\infty]$ be fixed. For any $k\geq 1$, by Lemma \ref{dualnonexplosive}, one has for all $x\geq 0$, $P_t^{(k)}e_x(z):=\mathbb{E}_z(e^{-xZ_t^{(k)}})=\mathbb{E}_x(e^{-zU_t^{(k)}})$ where $(U_t^{(k)},t\geq 0)$ is defined in Lemma \ref{approxU0}. Since, for any $t\geq 0$, $U_t^{(k)}$ converges almost-surely towards $U_t^{(\infty)}$ as $k$ goes to infinity, then $\underset{k\rightarrow \infty}{\lim}\mathbb{E}_z(e^{-xZ_t^{(k)}})=\mathbb{E}_x(e^{-zU_t^{(\infty)}})=\mathbb{E}_x(e^{-zU_t^{0}})$. Therefore $p_t^{(k)}(z,\cdot)$ converges weakly as $k$ goes to $\infty$ towards some probability $p_t(z,\cdot)$ over $[0,\infty]$ satisfying 
\begin{align*}
P_te_x(z)&:=\int_{[0,\infty]}e^{-xy}p_t(z,\ddr y)=\mathbb{E}_{x}[e^{-zU_t^{0}}] \text{ for any } z\in [0,\infty),\\
P_te_x(\infty)&:=\int_{[[0,\infty]} e^{-xy} p_t(\infty,\ddr y)=\mathbb{P}_{x}(U_t^{0}=0).
\end{align*}
Since $z\in [0,\infty]\mapsto P_te_x(z)$ is continuous, then by Stone-Weierstrass, $P_tf$ is continuous for any function $f\in C_b([0,\infty])$. We stress that at this stage, we do not know that $(P_t,t\geq 0)$ forms a semigroup. We establish now that for any $x\geq 0$, $(P_t^{(k)}e_x,k\geq 1)$ converges uniformly towards $P_te_x$. For any $x\geq 0$, define $\phi_k(x,t):=||P_t^{(k)}e_{x}-P_te_{x}||_{\infty}$. Recall that $U_t^{(k)}\leq U_t^{(\infty)}$ for any $t\geq 0$,   $\mathbb{P}_x$-almost-surely. Therefore, for any $z\in [0,\infty]$ and $t\geq 0$,  

\begin{align*}
&\mathbb{E}_{x}\left[e^{-zU_t^{(k)}}-e^{-zU_t^{(\infty)}}\right]\\
&=\mathbb{E}_{x}\left[(e^{-zU_t^{(k)}}-e^{-zU_t^{(\infty)}})\mathbbm{1}_{\{U_t^{(k)}<U_t^{(\infty)}\}}\right]\\
&=\mathbb{E}_{x}\left[(e^{-zU_t^{(k)}}-e^{-zU_t^{(\infty)}})\mathbbm{1}_{\{0<U_t^{(k)}<U_t^{(\infty)}\}}\right]+\mathbb{E}_{x}\left[(e^{-zU_t^{(k)}}-e^{-zU_t^{(\infty)}})\mathbbm{1}_{\{U_t^{(k)}=0,U_t^{(\infty)}>0\}}\right]\\
&\leq \mathbb{E}_{x}\left[(e^{-zU_t^{(k)}}-e^{-zU_t^{(\infty)}})\mathbbm{1}_{\{0<U_t^{(k)}<U_t^{(\infty)}\}}\right]+\mathbb{P}_x(\tau^k\leq t<\tau^\infty).
\end{align*}
Since $\tau^k\underset{k\rightarrow \infty}{\longrightarrow} \tau^{\infty}$ a.s. then $\mathbb{P}_x(\tau^k\leq t<\tau^\infty)\underset{k\rightarrow \infty}\longrightarrow 0$.
Now, for any $x\geq 0$, $$\sup_{z\in [0,\infty]}\mathbb{E}_{x}\left[(e^{-zU_t^{(k)}}-e^{-zU_t^{(\infty)}})\mathbbm{1}_{\{0<U_t^{(k)}<U_t^{(\infty)}\}}\right] \leq \mathbb{E}_{x}\left[\sup_{z\in [0,\infty]}\left(e^{-zU_t^{(k)}}-e^{-zU_t^{(\infty)}}\right)\mathbbm{1}_{\{0<U_t^{(k)}<U_t^{(\infty)}\}} \right].$$
The function $z\mapsto e^{-zU_t^{(k)}}-e^{-zU_t^{(\infty)}}$ reaches its maximum at $z_k=\frac{\log U_t^{(\infty)}-\log U_t^{(k)}}{U_t^{(\infty)}-U_t^{(k)}}$. Moreover, on the event $\{0<U_t^{(k)}<U_t^{(\infty)}\}$, $z_kU_t^{(\infty)}\underset{k\rightarrow\infty}{\longrightarrow} 1$ and $z_kU_t^{(k)}\underset{k\rightarrow\infty}{\longrightarrow} 1$ almost-surely, thus by Lebesgue's theorem $$\mathbb{E}_x\left[(e^{-z_kU_t^{(k)}}-e^{-z_kU_t^{(\infty)}})\mathbbm{1}_{\{0<U_t^{(k)}<U_t^{(\infty)}\}}\right]\underset{k\rightarrow \infty}{\longrightarrow} 0.$$
Hence, $\phi_{k}(x,t) \underset{k\rightarrow \infty}{\longrightarrow} 0$ and by Stone-Weierstrass, for any $f\in C_b([0,\infty])$ $||P_t^{(k)}f-P_tf||_\infty \underset{k\rightarrow \infty}{\longrightarrow} 0$. As a first consequence of the uniform convergence, we show now that $(P_t,t\geq 0)$ is a semigroup. Let $s,t\geq 0$. Let $g\in C_b([0,\infty])$ and $(g_k,k\geq 1)$ such that $g_k\underset{k\rightarrow \infty}{\longrightarrow} g$ uniformly. Then,
\begin{align*}
||P_t^{(k)}g_k-P_tg||_\infty &\leq ||P_t^{(k)}g_k-P^{(k)}_tg||_\infty+||P_t^{(k)}g-P_tg||_\infty\\
&\leq ||g_k-g||_\infty+||P_t^{(k)}g-P_tg||_\infty,
\end{align*}
where we have used in the second inequality that $P_t^{(k)}$ is a contraction. The upper bound goes to $0$ as $k$ goes to $\infty$ by the uniform convergence. Let $f\in C_b([0,\infty])$ and apply the last convergence to $g_k:=P_s^{(k)}f$ and $g=P_sf$, one has, by the semigroup property of $(P_t^{(k)},t\geq 0)$.
$$P_{t+s}f=\underset{k\rightarrow \infty}{\lim} P_{t+s}^{(k)}f=\underset{k\rightarrow \infty}{\lim} P_{t}^{(k)}g_k=P_{t}P_sf.$$ 
Therefore $(P_t,t\geq 0)$ is a semigroup on $C_b([0,\infty])$. As in Lemma \ref{Zextendedentrance}, we see that it is continuous at $0$. Lastly, since the convergence of semigroups is uniform in $C_b([0,\infty])$, one invokes Theorem 2.5 p167 in \cite{EthierKurtz} to claim that the sequence of processes $(Z_t^{(k)},t\geq 0)$ converges weakly (in the Skorokhod topology) towards a c\`adl\`ag Markov process $(Z_t,t\geq 0)$ with semigroup $(P_t,t\geq 0)$.
\end{proof}
\begin{lemma}\label{weakconvexit}  Assume $\frac{2\lambda}{c}\geq 1$, the sequence $((Z^{(k)}_t)_{t\geq 0}, k\geq 1)$ converges weakly towards a c\`adl\`ag Feller process $(Z_t, t\geq 0)$ valued in $[0,\infty]$ such that for all $z\in [0,\infty]$, all $t\geq 0$, and all $x\in (0,\infty)$ $$\mathbb{E}_{z}[e^{-xZ_t}]=\mathbb{E}_{x}[e^{-zU_{t}}],$$
where $(U_t,t\geq 0)$ is the $\Psi$-generalized Feller diffusion (\ref{generalizedFellerdiffusionequation}) with $0$ entrance.
\end{lemma}
\begin{proof}
The only difference with the proof above lies in the fact that we have to separate the cases $x>0$ and $x=0$, since $0$ is an entrance boundary. By Lemma \ref{approxU0}, $(U_t^{(k)},t\geq 0)$ converges towards $(U^{(\infty)}_t,t\geq 0)$ in a pointwise sense $\mathbb{P}_x$-almost-surely, for any $x>0$ and $(U_t^{(\infty)},t\geq 0)$ has the same law as the diffusion $(U_t,t\geq 0)$ solving (\ref{generalizedFellerdiffusionequation}) with $0$ entrance. The limiting semigroup is then given by $P_te_{x}(z)=\mathbb{E}_{x}[e^{-zU_{t}}]$ for any $z\in [0,\infty]$ with $x>0$. The case $x=0$ is trivial since for any $z\in [0,\infty]$ and any $k\geq 1$, one has $\mathbb{E}_z[e^{-0.Z_t^{(k)}}]=\mathbb{E}_{0}[e^{-zU_t^{(k)}}]=1$ and so $P_te_{0}(z)=1$ for any $z\in [0,\infty]$. The rest of the proof is similar to the one above. By replacing $\tau^{\infty}$ by $\infty$, since $0$ is inaccessible, and using the fact that $\mathbb{P}_x(\tau^{k}\leq t)\underset{k\rightarrow \infty}{\longrightarrow} 0$ for any fixed $t$, we see that $\phi_k(x,t) \underset{k\rightarrow \infty}{\longrightarrow} 0.$
\end{proof}
\noindent The next lemma ensures that the processes $(Z_t,t\geq 0)$ defined above are extensions of the minimal logistic CSBP.
\begin{lemma}[Extension]\label{extension}
The limiting processes $(Z_t,t\geq 0)$ in Lemma \ref{weakconv} and Lemma \ref{weakconvexit} stopped at $\zeta_\infty$ have the same law as $(\zmin_t,t\geq 0)$.
\end{lemma}

\begin{proof} By definition, $\infty$ is absorbing for $(Z_{t\wedge \zeta_\infty},t\geq 0)$. By the semigroup representation obtained in Lemma \ref{weakconv} and Lemma \ref{weakconvexit}, we see that $0$ is absorbing for the process $(Z_t,t\geq 0)$. Therefore it is absorbing for the process $(Z_{t\wedge \zeta_\infty},t\geq 0)$. It remains only to see if $(Z_{t\wedge \zeta_\infty},t\geq 0)$ satisfies the martingale problem (\textbf{MP}). If this holds true, Lemma \ref{wellposed} will ensure that $(Z_{t\wedge \zeta_\infty}, t\geq 0)$ has the same law as $(\zmin_t,t\geq 0)$. Denote by $\mathscr{L}^{(k)}$ the generator of $(Z_t^{(k)},t\geq 0)$ as defined in Equation (\ref{generatormin}). For any function $f\in C^2_c((0,\infty))$, one has 

\begin{align*}
||\mathscr{L}^{(k)}f -\mathscr{L}f||_\infty&=\sup_{z\in [0,\infty[}\left| f(z+k)(\bar{\pi}(k)+\lambda)-\int_{k}^{\infty}f(z+u)\pi(\ddr u)\right|\\
&\leq 2(\bar{\pi}(k)+\lambda)\!\!\!\sup_{z\in [k,\infty[}\!\!|f(z)| \underset{k\rightarrow \infty}{\longrightarrow} 0.
\end{align*}
For any $z\in (0,\infty)$ and any $k\geq 1$, under $\mathbb{P}_z$, the process $\left(f(Z^{(k)}_t)-\int_{0}^{t}\mathscr{L}^{(k)}f(Z_s^{(k)})\ddr s, t\geq 0\right)$
is a martingale. The convergence above being uniform, Lemma 5.1 page 196 in \cite{EthierKurtz} entails that the process
\[\left(f(Z_{t})-\int_{0}^{t}\mathscr{L}f(Z_s)\ddr s, t\geq 0\right)\] is a martingale. By \cite[Corollary III.3.6]{MR1725357}, the latter process stopped at time $\zeta_\infty$, $$\left(f(Z_{t\wedge \zeta_\infty})-\int_{0}^{t\wedge \zeta_\infty}\mathscr{L}f(Z_s)\ddr s, t\geq 0\right)$$ is a martingale. Now, since $f$ has a compact support then for any $t\geq 0$, $$\int_{0}^{t}\mathscr{L}f(Z_{s\wedge \zeta_\infty})\ddr s=\int_{0}^{t\wedge \zeta_\infty}\mathscr{L}f(Z_{s\wedge \zeta_\infty})\ddr s+\int_{t\wedge \zeta_\infty}^t\mathscr{L}f(Z_{s\wedge \zeta_\infty})\ddr s=\int_{0}^{t\wedge \zeta_\infty}\mathscr{L}f(Z_{s\wedge \zeta_\infty})\ddr s.$$
Indeed, in the first equality above, either $t\leq \zeta_\infty$ and the second integral  vanishes or $t>\zeta_\infty$ and $\mathcal{L}f(Z_{s\wedge \zeta_\infty})=0$ for any $s>\zeta_\infty$, since $f$ has a compact support. We deduce that for any $f\in C^2_c((0,\infty))$, the process
$$\left(f(Z_{t\wedge \zeta_\infty})-\int_{0}^{t}\mathscr{L}f(Z_{s\wedge \zeta_\infty})\ddr s, t\geq 0\right)$$ is a martingale. Therefore the process $(Z_{t\wedge \zeta_\infty},t\geq 0)$ has the same law as  $(\zmin_t,t\geq 0)$. Note that by Lemma \ref{explosionlemma}, since $\mathcal{E}<\infty$, then $(\zmin_t,t\geq 0)$ explodes with a positive probability. This ensures that  $\infty$ is accessible for the process $(Z_t,t\geq 0)$.
\end{proof}
\begin{lemma}[Reflecting boundary] \label{reflectinglemmaZ} 
Assume $\mathcal{E}<\infty$ and $0\leq\frac{2\lambda}{c}<1$. The boundary $\infty$ of the Feller process $(Z_t, t\geq 0)$  is instantaneous regular reflecting. Moreover, $\mathbb{P}_\infty$-almost-surely, $(Z_t,t\geq 0)$ enters $(0,\infty)$ continuously.
\end{lemma}
\begin{proof}
Recall that we assume $\mathcal{E}<\infty$ and $\frac{2\lambda}{c}<1$. Accordingly to Lemma \ref{extension}, $\infty$ is accessible. Moreover, for every $t\geq 0$, $$\mathbb{E}_{\infty}[e^{-xZ_t}]=\mathbb{P}_{x}(U^0_{t}=0)=\mathbb{P}_{x}(\tau_0\leq t)>0$$ since  $\frac{2\lambda}{c}<1$. Thus, $\infty$ is a regular boundary. We verify now that $\infty$ is reflecting. Since $(U^0_t,t\geq 0)$ has $0$ regular absorbing then  $$\mathbb{P}_z(Z_t<\infty)=\underset{x\rightarrow 0 \atop x>0}{\lim} \mathbb{E}_z(e^{-xZ_t})=\underset{x\rightarrow 0 \atop x>0}{\lim} \mathbb{E}_x(e^{-zU^{0}_t})=\mathbb{E}_{0+}(e^{-zU^0_t})=1.$$ This ensures that the Lebesgue measure of the set of times at which the process is at $\infty$ is zero.  In other words, $\infty$ is reflecting.  We show now that $\infty$ is instantaneous. Let $T=\inf\{t\geq 0, Z_t<\infty\}$. For any $t>0$, $\mathbb{P}_{\infty}(T\leq t)\geq \mathbb{P}_{\infty}(Z_t<\infty)=1$, then by letting $t$ to $0$, one has $\mathbb{P}_\infty(T=0)=1$. By right-continuity of $(Z_t,t\geq 0)$, we have $\mathbb{P}_\infty(\underset{t\rightarrow 0}{\lim}Z_t=\infty)=1$.
\end{proof}
\begin{lemma}[Exit boundary]\label{exitlemmaZ}
Assume $\frac{2\lambda}{c}\geq 1$. The boundary $\infty$ of the Feller process $(Z_t, t\geq 0)$  is an exit.
\end{lemma}
\begin{proof}
Firstly, by Lemma \ref{extension}, we know that the process $(Z_t, t\geq 0)$ explodes with positive probability (note that $\lambda>0$, so that it will explode by a jump with positive probability). Moreover, by Lemma \ref{weakconv}, letting $z\rightarrow \infty$, we obtain $\mathbb{E}_\infty[e^{-xZ_t}]=\mathbb{P}_x(U_t=0)=0$ since $0$ is an entrance boundary of $U$. Therefore, $(Z_t,t\geq 0)$ get absorbed at $\infty$. Note that for any $t>0$, $\mathbb{P}_{z}(\zeta_\infty<t)=1-\mathbb{E}_{0+}[e^{-xU_t}].$
\end{proof}
The proof of Theorem \ref{regularboundarytheo} now follows by combining Lemmas \ref{weakconv}, \ref{extension}, \ref{reflectinglemmaZ}.  Theorem \ref{exitboundarytheo} follows from Lemmas \ref{weakconvexit}, \ref{extension} and \ref{exitlemmaZ}. In order to understand the long-term behavior of the extended process $(Z_t,t\geq 0)$, we establish now Corollary \ref{zeroboundary}, Corollary  \ref{StationarityZ} and Theorem \ref{longtermtheo} in the case $\mathcal{E}<\infty$.  The arguments for Corollary \ref{zeroboundary} and  Corollary \ref{StationarityZ} are exactly the same as those in Lemma \ref{0access} and Lemma \ref{stationarityandalmostsureabsorptionEinfini}
for the case $\mathcal{E}=\infty$ (but with $\lambda\geq 0$). Indeed, according to Lemma \ref{stationarylawlemma}, the dual diffusion $(U^{0}_t,t\geq  0)$ can leave the interval $(0,\infty)$ by $\infty$ only when the condition \textbf{(A)} is not satisfied. We show Theorem \ref{longtermtheo} in the case $\mathcal{E}<\infty$.
\begin{lemma} Assume $0\leq \frac{2\lambda}{c}<1$, $\mathcal{E}<\infty$ and $\Psi(z)\geq 0$ for some $z>0$, then the boundary $0$ is attracting almost-surely: $Z_t\underset{t\rightarrow \infty}{\longrightarrow} 0$ a.s.  If $\int^{\infty}\frac{\ddr z}{\Psi(z)}=\infty$ then $Z_t>0$ for all $t\geq 0$ a.s, and if $\int^{\infty}\frac{\ddr z}{\Psi(z)}<\infty$, then for any $z\in [0,\infty]$, $\mathbb{P}_z(\zeta_0<\infty)=1$.
\end{lemma}
\begin{proof} Let $a>0$. Define  $\zeta_{\infty}^{(0)}:=0$, $\zeta_a^{(n)}:=\inf\{t>\zeta_{\infty}^{(n-1)}; Z_t\leq a\}$ and $\zeta_{\infty}^{(n)}:=\inf\{t>\zeta_a^{(n)}; Z_t=\infty\}$. By Lemma \ref{stationarylawlemma}-1), since $-\Psi$ is not the Laplace exponent of a subordinator then $U_t^{0}\underset{t\rightarrow \infty}{\longrightarrow} 0$ and by duality $\mathbb{E}_z[e^{-xZ_t}]\underset{t\rightarrow \infty}{\longrightarrow} 1$. Therefore $Z_t\underset{t\rightarrow \infty}{\longrightarrow} 0$ in probability.  Since $\{Z_t\leq a\}\subset \{\zeta^{(1)}_{a}\leq t\}$ and $\mathbb{P}_z(Z_t\leq a)\underset{t\rightarrow \infty}{\longrightarrow} 1$ then $\zeta^{(1)}_{a}<\infty$ a.s. Applying the Markov property at $\zeta_a^{(n)}$, we see that the processes $(Z_{(t+\zeta_a^{(n)})\wedge \zeta_\infty^{(n)}},t\geq 0)$ are independent and with the same law as the minimal process starting from $a$. Set $E_n=\{Z_{(t+\zeta_a^{(n)})\wedge \zeta_\infty^{(n)}} \underset{t\rightarrow \infty}{\longrightarrow} 0\}$ for any $n\geq 1$. By Lemma \ref{longtermminimal}, $\mathbb{P}_{z}(E_n)=\mathbb{P}_a(\sigma_0<\infty)>0$. By independence, $\mathbb{P}_z\left(\bigcap_{n=1}^{\infty}E^{c}_n\right)=0$ and therefore $Z_t\underset{t\rightarrow \infty}{\longrightarrow} 0$ a.s. Assume $\int^{\infty}\frac{\ddr z}{\Psi(z)}=\infty$, since $\infty$ is a natural boundary of $(U_t^{0},t\geq 0)$ therefore by duality, $\mathbb{P}_z(\zeta_0\leq t)=\mathbb{E}_{\infty}(e^{-zU_t^{0}})=0$ for any $t\geq 0$. Thus $Z_t>0$ for all $t\geq 0$ a.s. Now if $\int^{\infty}\frac{\ddr z}{\Psi(z)}<\infty$, $\infty$ is an entrance of $(U_t^{0},t\geq 0)$ and by duality $\mathbb{P}_z(\zeta_0\leq t)=\mathbb{E}_{\infty}(e^{-zU_t^{0}})>0$. By Lemma \ref{stationarylawlemma}-1), $U_t^{0}\underset{t\rightarrow \infty}{\longrightarrow} 0$ a.s and thus $\mathbb{P}_z(\zeta_0\leq t)\underset{t\rightarrow \infty}{\longrightarrow} \mathbb{P}_z(\zeta_0< \infty)=1$. This proves Theorem \ref{longtermtheo}-1.
\end{proof}
It remains to study the process when $\frac{2\lambda}{c}\geq 1$. 
\begin{lemma}\label{extinctionexit} Assume $\frac{2\lambda}{c}\geq 1$.  \begin{itemize}
\item[a)] If $-\Psi$ is the Laplace exponent of a subordinator then the process is absorbed at $\infty$ almost-surely. 
\item[b)] If $-\Psi$ is not the Laplace exponent of a subordinator then the process tends to $0$ with probability: 
$$\mathbb{P}_z(Z_t\underset{t\rightarrow \infty}{\longrightarrow} 0)=\frac{\int_{0}^{\infty}e^{-zu}\frac{1}{u}\exp\left(-\int_{\theta}^{u}\frac{2\Psi(v)}{cv}\ddr v\right)\ddr u}{\int_{0}^{\infty}\frac{1}{u}\exp\left(-\int_{\theta}^{u}\frac{2\Psi(v)}{cv}\ddr v\right)\ddr u}\in (0,1).$$
On this event, the process get absorbed at $0$ if and only if $\int^\infty\frac{\ddr u}{\Psi(u)}<\infty$.
\end{itemize}
\end{lemma}
\begin{proof} By Lemma \ref{exitlemmaZ}, the process $(Z_t,t\geq 0)$ has $0$ and $\infty$  has absorbing boundary. By Lemma \ref{extension}, it satisfies $(\textbf{MP})$, therefore $(Z_t,t\geq 0)$ has the same law as $(\zmin_t,t\geq 0)$. Proof of a). Assume that $-\Psi$ is the Laplace exponent of a subordinator, then by Lemma \ref{irreducibleclass}, explosion is almost-sure. Proof of b). If now $-\Psi$ is not the Laplace exponent of a subordinator, then by Lemma \ref{longtermminimal} $$\mathbb{P}_z(Z_t\underset{t\rightarrow \infty}{\longrightarrow} 0)=\mathbb{P}_z(\zeta_\infty=\infty)=\mathbb{P}_z(\sigma_0<\infty)>0.$$ 
The process $(Z_t,t\geq 0)$ hit $0$ with positive probability if and only if $\int^{\infty}\frac{\ddr z}{\Psi(z)}<\infty$. When $\int^{\infty}\frac{\ddr z}{\Psi(z)}<\infty$ then $\{\zeta_0<\infty\}=\{\zeta_\infty=\infty\}=\{\zeta_0<\zeta_\infty\}$. 
\end{proof}
Lemma \ref{extinctionexit}-a) gives part 2) of  Theorem \ref{longtermtheo}, and b) gives part 3). 
As mentioned in the introduction, our extension of the minimal process has been done without using excursions theory. In particular, we have not discussed the existence of a local time at $\infty$. We end this article by showing that when $\infty$ is regular reflecting, the process immediately returns to $\infty$ after leaving it. Such a boundary is said to be regular for itself and standard theory, see for instance Bertoin \cite[Chapter IV]{Ber96}, would then ensure the existence of a local time. 
\begin{proposition}
Assume $\mathcal{E}<\infty$ and $0\leq\frac{2\lambda}{c}<1$. Set $S_\infty:=\inf\{t>0; Z_t=\infty\}$, then $\mathbb{P}_\infty(S_\infty=0)=1.$
\end{proposition}
\begin{proof}
Assume $\mathcal{E}<\infty$. 
Under $\mathbb{P}_z$ with $z\in (0,\infty)$, $S_\infty$ and $\zeta_\infty$ have the same law. We first prove that $\mathbb{E}_{z}[S_\infty,S_\infty<\infty]\underset{z\rightarrow \infty}{\longrightarrow} 0.$ By the time-change studied in Section \ref{existence}, $\mathbb{E}_{z}[S_\infty,S_\infty<\infty]=\mathbb{E}_{z}\left[\int^{\infty}_0\frac{\ddr s}{R_s}, \sigma_0=\infty\right]$. Let $\epsilon>0$. 
Set $\xi_a:=\sup\{s>0: R_s<a\}$. Since $\mathcal{E}<\infty$ then $R_s\underset{s\rightarrow \infty}{\longrightarrow} \infty$ on the event $\{\sigma_0=\infty\}$ and thus $\xi_a<\infty$ a.s. Moreover, on the event $\{\sigma_0=\infty\}$, $\xi_a\underset{a\rightarrow 0}{\longrightarrow} 0$ and one can choose $a$ small enough such that $\int_{0}^{\xi_a}\frac{\ddr s}{R_s}\leq \epsilon/2$ a.s. 
From the inequality (\ref{meanexplosion}), we see by applying Lebesgue's theorem that for any $b>0$, 
$$\mathbb{E}_{z}\left[\int_{0}^{\infty}\frac{1-e^{-bR_s}}{R_s}\ddr z,\sigma_0=\infty\right]\underset{z\rightarrow \infty}{\longrightarrow} 0.$$
 For any $s\geq \xi_a$, $R_s\geq a$ and therefore $1-e^{-bR_s}\geq 1-e^{-ba}$. We deduce that 
\begin{align*}
\mathbb{E}_{z}\left[\int_{0}^{\infty}\frac{1-e^{-b R_s}}{R_s}\ddr z,\sigma_0=\infty\right]&\geq \mathbb{E}_{z}\left[\int_{\xi_a}^{\infty}\frac{1-e^{-b R_s}}{R_s}\ddr z,\sigma_0=\infty\right]\\
&\geq (1-e^{-ba})\mathbb{E}_z\left[\int_{\xi_a}^{\infty}\frac{\ddr s}{R_s},\sigma_0=\infty\right].
\end{align*}
Therefore, for $z$ large enough, $\mathbb{E}_z\left[\int_{\xi_a}^{\infty}\frac{\ddr s}{R_s},\sigma_0=\infty\right]\leq \epsilon/2$ and combining this with $\mathbb{E}_z\left[\int_{0}^{\xi_a}\frac{\ddr s}{R_s},\sigma_0=\infty\right]\leq \frac{\epsilon}{2}$,  we obtain
$$\mathbb{E}_z\left[\int_{0}^{\infty}\frac{\ddr s}{R_s},\sigma_0=\infty\right]\leq \epsilon.$$
By the Markov inequality,
$\mathbb{P}_z(S_\infty>t,S_\infty<\infty)\leq \frac{1}{t}\mathbb{E}_z[S_\infty,S_\infty<\infty]$ and thus \begin{equation}\label{limitz}\mathbb{P}_z(S_\infty>t,S_\infty<\infty)\underset{z\rightarrow \infty}{\longrightarrow} 0. \end{equation} Let $s>0$. By the Markov property at time $s$,
\begin{align*}
\mathbb{E}_{\infty}\left[\mathbbm{1}_{\{S_\infty>t+s\}}\mathbbm{1}_{\{S_\infty<\infty\}}\right]&=\mathbb{E}_\infty\left[\mathbb{E}[\mathbbm{1}_{\{S_\infty>t+s\}}\mathbbm{1}_{\{S_\infty<\infty\}}\lvert \mathcal{F}_s]\right]\\
&\leq \mathbb{E}_\infty\left[\mathbb{E}_{Z_s}[\mathbbm{1}_{\{S_\infty>t\}}\mathbbm{1}_{\{S_\infty<\infty\}}]\right]
\end{align*} 
By right continuity, $Z_s\underset{s\rightarrow 0}{\longrightarrow} \infty$ a.s under $\mathbb{P}_\infty$, and by (\ref{limitz}),
$\mathbb{E}_{Z_s}[\mathbbm{1}_{\{S_\infty>t\}}\mathbbm{1}_{\{S_\infty<\infty\}}]\underset{s\rightarrow 0}{\longrightarrow} 0$ a.s under $\mathbb{P}_\infty$. By Lebesgue's theorem, we conclude that, for any $t>0$,
$\mathbb{P}_\infty(S_\infty>t,S_\infty<\infty)=0.$ We now verify that $\mathbb{P}_\infty(S_\infty=\infty)=0$. By the time-change $\{S_\infty=\infty\}=\{\sigma_0<\infty\}$ under $\mathbb{P}_z$ for any $z\in [0,\infty)$. By the Markov property, $$\mathbb{P}_{\infty}(S_\infty=\infty)=\mathbb{E}_\infty(\mathbb{E}(\mathbbm{1}_{S_\infty=\infty}|\mathcal{F}_s))=\mathbb{E}_\infty(\mathbb{P}_{Z_s}(S_\infty=\infty)).$$
On the one hand, if $-\Psi$ is the Laplace exponent of a subordinator then $\mathbb{P}_{z}(S_\infty=\infty)=0$ for any $z<\infty$ and  since $Z_s<\infty$ $\mathbb{P}_{\infty}$-a.s, then $\mathbb{P}_{\infty}(S_\infty=\infty)=0$. On the other hand, if $-\Psi$ is not the Laplace exponent of a subordinator, we see in Lemma \ref{longtermminimal} that 
$\mathbb{P}_z(\sigma_0<\infty)\underset{z\rightarrow \infty}{\longrightarrow} 0$.
Since $Z_s\underset{s\rightarrow 0}{\longrightarrow} \infty$ $\mathbb{P}_{\infty}$-a.s then $\mathbb{P}_{\infty}(S_\infty=\infty)=0.$ This entails that $\mathbb{P}_\infty(S_\infty>t,S_\infty<\infty)=\mathbb{P}_\infty(S_\infty>t)=0$ for any $t>0$ and thus $\mathbb{P}_\infty(S_\infty=0)=1$.
\end{proof}
\noindent The last proposition leads to the natural question of characterizing the inverse local time. A related question is to see if one can continue Table \ref{correspondance} to the case where the boundary $\infty$ of $Z$ is regular absorbing. This requires us to work with the reflected diffusion and seems to require a deeper analysis of the duality at the level of generators. We intend to study these questions in future work. 
\section{Appendix}
\noindent We first prove Proposition \ref{example}. Recall 
$$\mathcal{E}=\int_{0}^{\theta}\frac{1}{x}\exp \left({\frac{2}{c}\int_x^\theta \frac{\imf{\Psi}{u}}{u}\, \ddr u}\right)\ddr x \text{ and } \mathcal{E}'=\int_{0}^\theta \frac{1}{x}\exp\left(-\frac{2}{c}\int_{1}^{\infty}e^{-xv}\frac{\bar{\pi}(v)}{v}\ddr v\right)\ddr x.$$
\begin{proof}[Proof of Proposition \ref{example}] 
From the L\'evy-Khintchine form (\ref{Levykhintchine}) of $\Psi$, one has for any $u>0$, 
\begin{equation}\label{psi(u)/u}
\frac{\Psi(u)}{u}=-\int_{0}^{\infty}e^{-uv}\bar{\pi}(v\vee 1)\ddr v+\gamma+\frac{\sigma^2}{2}u+ \int_{0}^1\pi((v,1))(1-e^{-uv})\ddr v
\end{equation}
where $v\vee 1=\max(v,1)$. For any $z>0$, set $H(z):=-\int_{z}^{\theta}\frac{\Psi(u)}{u}\ddr u$. One can readily check from the above expression that as $z$ goes to $0$, $$H(z)
=\int_{1}^{\infty}e^{-zv}\frac{\bar{\pi}(v)}{v}\ddr v+o_{\theta}(1)+c_\theta.$$  The last term $o_{\theta}(1)+c_\theta$ will not play any role in the convergence of $\mathcal{E}$. Therefore, $\mathcal{E}$ and $\mathcal{E}'$ have the same nature. Without loss of generality, we assume $\pi((0,1))=0$, $\sigma=0$ and $\gamma=0$. Define $\Phi$ such that $\Phi(z)/z=H(z)$. The function $\Phi$ is the Laplace exponent of a driftless subordinator with L\'evy measure $\nu$ whose tail is $\bar{\nu}(v)= \frac{\bar{\pi}(v)}{v}$ for all $v\geq 1$. Recall from \cite[Chapter III, Proposition 1]{Ber96} that there exists a universal constant $\kappa>0$ such that $$\frac{1}{\kappa}I(1/z)\leq \Phi(z)/z\leq \kappa I(1/z), \forall z>0, \text{ where } I(z):=\int_1^z \bar{\nu}(r)\ddr r.$$
Thus, we have $$ \frac{1}{\kappa} \int_{1}^{1/z}\frac{\bar{\pi}(u)}{u}\ddr u \leq 
H(z)\leq \kappa \int_{1}^{1/z}\frac{\bar{\pi}(u)}{u}\ddr u$$
Therefore 
$$\int_0^\theta\frac{1}{x}\exp\left(-\frac{2\kappa}{c}\int_{1}^{1/x}\frac{\bar{\pi}(u)}{u}\ddr u\right)\ddr x\leq \mathcal{E}\leq \int_0^\theta\frac{1}{x}\exp\left(-\frac{2}{c\kappa}\int_{1}^{1/x}\frac{\bar{\pi}(u)}{u}\ddr u\right)\ddr x.$$ 
\end{proof}
We now study the boundaries of $(U_t, t< \tau)$ solution to (\ref{generalizedFellerdiffusionequation}). This is a regular\footnote{in the sense irreducible} diffusion in $(0,\infty)$. We apply Feller's tests for which we refer to   Durrett's book \cite[Chapter 6, Section 6.2 and 6.5]{MR1398879}.
Denote by $s$ the scale function of $(U_t,t< \tau)$ and by $m$ the density of its speed measure. Fix $\theta \in (0,\infty)$. For any $x\in (0,\infty)$, $$s(x)=\int_{\theta}^{x}\exp\left(\int_{\theta}^{y}\frac{2\Psi(u)}{c u}\ddr u\right)\ddr y,\quad m(x)=\frac{1}{cx}\exp\left(-\int_{\theta}^{x}\frac{2\Psi(u)}{cu}\ddr u\right).$$
Denote by $M$ the antiderivative of $m$, $M(z)=\int_{\theta}^{z}m(x)\ddr x$. We establish Lemma \ref{entranceU} and Lemma \ref{stationarylawlemma}. 
\begin{lemma}[Lemma \ref{entranceU}-2]\label{entranceX} The boundary $\infty$ is inaccessible. It is an entrance boundary if $\int^{\infty}\frac{\ddr x}{\Psi(x)}<\infty$ and a natural one if $\int^{\infty}\frac{\ddr x}{\Psi(x)}=\infty$.
\end{lemma}
\begin{proof}
Feller's tests imply that $\infty$ is an entrance boundary if and only if 
$$I_{\infty}=\int_{\theta}^{\infty}(s(\infty)-s(x))m(x)\ddr x=\int_{\theta}^{\infty}\frac{\ddr z}{z}\int_{z}^{\infty}\exp\left(\int_{z}^{x}\frac{2\Psi(u)}{cu}\ddr u\right)\ddr x=\infty$$ 
and $$J_\infty=\int_{\theta}^{\infty}(M(\infty)-M(z))s(z)\ddr z=\int_\theta^{\infty}\ddr z\int_{z}^{\infty}\frac{1}{x}\exp\left(-\int_{z}^{x}\frac{2\Psi(u)}{cu}\ddr u\right)\ddr x<\infty$$
If $-\Psi$ is not the Laplace exponent of a subordinator then there exists $z_0\geq 0$, such that for all $u\geq z_0$, $\Psi(u)\geq 0$ and then $I_\infty=\infty$. If $-\Psi$ is the Laplace exponent of a subordinator then $\underset{z\rightarrow \infty}{\lim} \Psi(z)/z=:\textbf{d}<0$. Therefore, for any $\epsilon>0$, there exists $z_0$ such that if $u\geq z_0$ then $\Psi(u)/u>\textbf{d}-\epsilon$. This entails: $I_\infty\geq \int_{\theta}^{\infty}\frac{e^{\frac{\textbf{d}-\epsilon}{c}z}}{z}e^{-\frac{\textbf{d}-\epsilon}{c}z}\ddr z=\infty$. We show now that $J_\infty<\infty$ if and only if $\int^{\infty}\frac{\ddr u}{\Psi(u)}<\infty$. Define $Q(x)=\int_{\theta}^{x}\frac{2\Psi(u)}{cu}\ddr u$, one has 
\begin{equation}\label{J} J_\infty=\int_{\theta}^{\infty}e^{Q(z)}\int_{z}^{\infty}\frac{e^{-Q(x)}}{cx}\ddr x\ddr z.
\end{equation}
Since any mechanism $\Psi$ is convex, $z\mapsto \Psi(z)/z$ is non-decreasing. Assume $\int^{\infty}\frac{\ddr z}{\Psi(z)}<\infty$. Note that this ensures that $-\Psi$ is not the Laplace exponent of a subordinator and thus $\Psi$ is non-decreasing over $(\rho, \infty)$ where $\rho$ is the largest root of $\Psi$. Moreover, since $\int^{\infty}\frac{\ddr z}{\Psi(z)}<\infty$ then there exists $v$ such that $\frac{\Psi(v)}{v}\geq 1$. Since the map $z\mapsto \Psi(z)/z$ is non-decreasing, therefore  for all $z\geq v$, $\frac{\Psi(z)}{z}\geq 1$ and $Q(\infty)=\infty$. For all $z\geq \theta>\rho$, by the mean value theorem there exists $\xi_{z}\geq z$ such that
$$\int_{z}^{\infty}\frac{\Psi(x)}{cx}e^{-Q(x)}\frac{\ddr x}{\Psi(x)}=\frac{1}{\Psi(\xi_z)}\left[-e^{-Q(x)}\right]_{x=z}^{x=\infty}=\frac{1}{\Psi(\xi_z)}e^{-Q(z)}.$$
Thus
$$J_\infty=\int_{\theta}^{\infty}e^{Q(z)}\int_{z}^{\infty}\frac{1}{cx}e^{-Q(x)}\ddr x\ddr z\leq \int_\theta^{\infty}\frac{\ddr z}{\Psi(z)}<\infty.$$
Conversely, assume $J_\infty<\infty$. Let $\varphi(z)=M(\infty)-M(z)=\int_{z}^{\infty}\frac{1}{cx}e^{-Q(x)}\ddr x$, one has $J_\infty=\underset{b\rightarrow \infty}{\lim} \int_{\theta}^{b}\varphi(z)e^{Q(z)}\ddr z.$ By integration by parts:
\begin{align*}
\int_{\theta}^{b}\varphi(z)e^{Q(z)}\ddr z&=\int_{\theta}^{b}\frac{\varphi(z)}{Q'(z)}e^{Q(z)}Q'(z)\ddr z=\left[\frac{\varphi(z)}{Q'(z)}e^{Q(z)}\right]_{\theta}^{b}-\int_{\theta}^{b}e^{Q(z)}\left(\frac{\varphi(z)}{Q'(z)}\right)'\ddr z.
\end{align*}
Moreover,
\begin{align*}\left(\frac{\varphi(z)}{Q'(z)}\right)'=\frac{\varphi'(z)Q'(z)-\varphi(z)Q''(x)}{Q'(z)^2}&=\frac{-\frac{1}{cz}e^{-Q(z)}\frac{\Psi(z)}{cz}-\varphi(z)\left(\frac{\Psi(z)}{z}\right)'}{\left(\frac{\Psi(z)}{cz}\right)^{2}}\\
&=-\frac{e^{-Q(z)}}{\Psi(z)}-\frac{\varphi(z)\left(\frac{\Psi(z)}{z}\right)'}{\left(\frac{\Psi(z)}{cz}\right)^{2}}\leq -\frac{e^{-Q(z)}}{\Psi(z)}.\end{align*}
The last inequality holds since $\left(\frac{\Psi(z)}{z}\right)'\geq 0$ and entails
\begin{align}\label{ineq}
\int_{\theta}^{b}\varphi(z)e^{Q(z)}\ddr z\geq \left[\frac{\varphi(z)}{Q'(z)}e^{Q(z)}\right]_{\theta}^{b}+\int_{\theta}^{b}\frac{\ddr z}{\Psi(z)}
\end{align}
Since $J_\infty<\infty$, then $\varphi(z)e^{Q(z)}\underset{z\rightarrow \infty}{\longrightarrow} 0$.  Provided that $\underset{z\rightarrow \infty}{\lim}Q'(z)\neq 0$, then (\ref{ineq}) gives 
\begin{align*}
\int_{\theta}^{\infty}\frac{\ddr z}{\Psi(z)}\leq 
J_\infty+\frac{\varphi(\theta)}{Q'(\theta)}e^{Q(\theta)}<\infty
\end{align*}
and the proof is complete. We show now that if $J_\infty<\infty$ then $\underset{z\rightarrow \infty}{\lim}Q'(z)\neq 0$. Assume by contradiction that for any $\epsilon>0$, there exists $z_0$ such that for all $z\geq z_0$, $Q'(z)=\frac{\Psi(z)}{cz}\leq \epsilon$. Recall $J_\infty=\int_\theta^{\infty}\ddr z\int_{z}^{\infty}\frac{1}{x}\exp\left(-\int_{z}^{x}\frac{\Psi(u)}{cu}\ddr u\right)\ddr x$, we get \begin{align*}
J_\infty&\geq \int_{z_0}^{\infty}\ddr z\int_{z}^{\infty}\frac{1}{x}e^{-\int_{z}^{x}\epsilon \ddr u}\ddr x=\int_{z_0}^{\infty}e^{\epsilon z}\ddr z\int_{z}^{\infty}\frac{1}{x}e^{-\epsilon x}\ddr x=\int_{z_0}^{\infty}\ddr x\left(\int_{z_0}^{x}\ddr ze^{\epsilon z}\right)\frac{e^{-\epsilon x}}{x}\\
&=\int_{z_0}^{\infty}\frac{e^{\epsilon x}-e^{\epsilon z_0}}{\epsilon}\frac{e^{-\epsilon x}}{x}\ddr x=\int_{z_0}^{\infty}\frac{1-e^{\epsilon(z_0-x)}}{\epsilon x}\ddr x=\infty.
\end{align*} 
\end{proof}
\begin{lemma}[Lemma \ref{entranceU}-1]\label{absorbing0} The boundary $0$ is an exit if
$\mathcal{E}=\int_{0}^\theta\frac{1}{x}\exp\left(\frac{2}{c}\int_{x}^{\theta}\frac{\Psi(u)}{u}\ddr u\right)\ddr x=\infty$, regular if $\mathcal{E}<\infty$ and $0\leq \frac{2\lambda}{c}<1$, and  an entrance if  $\frac{2\lambda}{c}\geq 1$.
\end{lemma}
\begin{proof}
Let $$I_0:=\int_0^{\theta}(s(x)-s(0))m(x)\ddr x=\frac{1}{c}\int_{0}^{\theta}\frac{\ddr x}{x}\int_{0}^{x}\exp\left(-\int_{y}^{x}\frac{2\Psi(z)}{cz}\ddr z\right)\ddr y.$$ We first show that $I_0<\infty$ if and only if $0\leq \frac{2\lambda}{c}<1$. This will provide that $0$ is accessible when $0\leq \frac{2\lambda}{c}<1$. Since $\Psi(0)=-\lambda$ and $\Psi$ is continuous at $0+$ then $-\frac{2}{c}\Psi(z)\underset{z\rightarrow 0+}{\longrightarrow} \frac{2\lambda}{c}.$ Assume $\frac{2\lambda}{c}<1$, and let $\epsilon$ such that $\frac{2\lambda}{c}+\epsilon<1$, there exists $\theta>0$ such that $-\frac{2}{c}\Psi(z)\leq \frac{2\lambda}{c}+\epsilon<1$ for all $z\leq \theta$. Therefore, $$|s(0)|=\int_{0}^{\theta}\exp\left(-\int_{y}^{\theta}\frac{2\Psi(z)}{cz}\ddr z\right)\ddr y\leq \int_{0}^{\theta}\exp\left(\int_{y}^{\theta}\frac{\frac{2\lambda}{c}+\epsilon}{z}\ddr z\right)\ddr y=c_\theta\int_{0}^{\theta}\frac{\ddr y}{y^{\frac{2\lambda}{c}+\epsilon}}<\infty.$$ 
We have  \begin{align*}
I_0\leq \frac{1}{c}\int_{0}^{\theta}\frac{\ddr x}{x}\int_{0}^{x}\exp\left(\int_{y}^{x}\frac{2\lambda/c+\epsilon}{z}\ddr z\right)\ddr y&=\frac{1}{c}\int_{0}^{\theta}\frac{\ddr x}{x}\int_{0}^{x}\exp\left[\left(2\lambda/c+\epsilon\right)(\ln(x)-\ln(y))\right]\ddr y\\
&=\frac{1}{c}\int_{0}^{\theta}x^{\frac{2\lambda}{c}+\epsilon-1}\ddr x\int_{0}^{x}\frac{1}{y^{\frac{2\lambda}{c}+\epsilon}}\ddr y\\
&=\frac{\theta}{c\left(1-\left(2\lambda/c+\epsilon\right)\right)}<\infty.
\end{align*}
Therefore when $0\leq \frac{2\lambda}{c}<1$ for all $y\in \mathbb{R}_+$ s.t. $y\neq x$, $\mathbb{P}_x(\tau_0<\tau_y)>0$. Assume $\frac{2\lambda}{c}\geq 1$ and define $\tilde{\Psi}(z)=\lambda+\Psi(z)$. Note that $\tilde{\Psi}$ is the mechanism of a CSBP with no killing term. One has
$$|s(0)|=\int_{0}^{\theta}\frac{1}{y^{2\lambda/c}}\exp\left(-\int_{y}^{\theta}\frac{2\tilde{\Psi}(z)}{cz}\ddr z\right)\ddr y.$$ 
If $\tilde{\Psi}(z)\geq 0$ for all $z\geq 0$, then $\underset{z\rightarrow 0}{\lim}\frac{\tilde{\Psi}(z)}{z}=\tilde{\Psi}'(0)\geq 0$ and $\int_{y}^{\theta}\frac{2\tilde{\Psi}(z)}{cz}\ddr z\underset{y\rightarrow 0}{\longrightarrow} \int_{0}^{\theta}\frac{2\tilde{\Psi}(z)}{cz}\ddr z<\infty$ thus $|s(0)|=\infty$. If $\tilde{\Psi}$ is non-positive in a neighbourhood of $0$ then $|s(0)|\geq \int_{0}^{\theta}\frac{\ddr y}{y^{2\lambda/c}}=\infty$, which implies that $0$ is inaccessible.  It remains to study $J_0$ defined by 
$$J_0:=\int_{0}(M(z)-M(0))s'(z)\ddr z$$
with $M(z)=\int^{z}m(x)\ddr x$, to see if the process can get out from $0$. When $0\leq \frac{2\lambda}{c}<1$, $|s(0)|<\infty$ and $J_0$ has the same nature as $|M(0)|=\mathcal{E}$. We deduce that $0$ is an exit if $\mathcal{E}=\infty$, regular if $0\leq \frac{2\lambda}{c}<1$ and $\mathcal{E}<\infty$.
Assume $\frac{2\lambda}{c}\geq 1$. One has 
\begin{align}\label{J_0}J_0&=\int_{0}^{\theta}\ddr z\int_{0}^{z}\frac{1}{x}\exp\left(\int_{x}^{z}\frac{2\Psi(u)}{cu}\ddr u \right)\ddr x=\int_{0}^{\theta}\ddr z\int_{0}^{z}\frac{1}{x}e^{-\frac{2\lambda}{c}\int_x^{z}\frac{\ddr u}{u}}e^{\int_{x}^{z}\frac{2\tilde{\Psi}(u)}{cu}\ddr u }\ddr x \nonumber \\
&=\int_{0}^{\theta}\frac{\ddr z}{z^{2\lambda/c}}\int_{0}^{z}x^{2\lambda/c-1}e^{\int_{x}^{z}\frac{2\tilde{\Psi}(u)}{cu}\ddr u}\ddr x.
\end{align}
 If $\tilde{\Psi}$ is non-positive in a neighbourhood of $0$, then $$J_0\leq \int_{0}^{\theta}\frac{\ddr z}{z^{2\lambda/c}}\int_{0}^{z}x^{2\lambda/c-1}\ddr x=\frac{\theta}{2\lambda/c}<\infty.$$
If $\tilde{\Psi}(z)\geq 0$ for all $z\geq 0$, then $\underset{z\rightarrow 0}{\lim} \tilde{\Psi}(z)/z\geq 0$ and then $0\leq \int_{x}^{z}\frac{2\tilde{\Psi}(u)}{cu}\ddr u\leq \int_{0}^{\theta}\frac{2\tilde{\Psi}(u)}{cu}\ddr u<\infty$, which provides $J_0<\infty$.  We deduce that when $\frac{2\lambda}{c}\geq 1$, then $0$ is an entrance.
\end{proof}
Recall that $(U_t^0,t\geq 0)$ denotes the $\Psi$-generalized Feller diffusion with $0$ as an exit (if $\mathcal{E}=\infty$) or a regular absorbing boundary (if $\mathcal{E}<\infty$ and $\frac{2\lambda}{c}<1$).
\begin{lemma}[Lemma \ref{stationarylawlemma}-1]\label{degenerate}
Assume there exists $z\geq 0$, such that $\Psi(z)\geq 0$ (-$\Psi$ is not the Laplace exponent of a subordinator), then for all $x\geq 0$, $\mathbb{P}_x(\underset{t\rightarrow \infty}{\lim} U^0_{t}=0)=1$.
\end{lemma}
\begin{proof}
Recall that if $s(\infty)=\infty$ then $\tau_\infty=\infty$ a.s. (see \cite[Theorem 3.3 page 220]{MR1398879}). By assumption  $-\Psi$ is not the Laplace exponent of a subordinator, we see from (\ref{psi(u)/u}) that this entails $\textbf{d}=\underset{z\rightarrow \infty}{\liminf}\frac{\Psi(z)}{z}>0$. One has $$s(\infty)=\int^{\infty}\exp\left(\int_{\theta}^{y}\frac{2\Psi(z)}{cz}\ddr z\right)\ddr y\geq \int^{\infty}\exp\left(\frac{2\textbf{d}}{c}y\right)\ddr y=\infty.$$
Thus, $\mathbb{P}_{x}(\tau_0<\tau_\infty)=1$ for all $x\geq 0$.
\end{proof}
If $\Psi(z)\leq 0$ for all $z\geq 0$, then $\Psi$ is of the form 
\begin{equation}\label{sub} \Psi(v)=-\lambda-\delta v-\int_{0}^{\infty}(1-e^{-vu})\pi(\ddr u)
\end{equation}
with $\delta\geq 0$ and $\int_{0}^{\infty}(1\wedge u)\pi(\ddr u)<\infty$. In this case,  $-\Psi$ is the Laplace exponent of a subordinator. Recall the condition $(\textbf{A})$:
$$\delta=0 \text{ and } \bar{\pi}(0)+\lambda\leq c/2. \qquad (\textbf{A})$$  
\begin{lemma}[Lemma \ref{stationarylawlemma}-2)]\label{nondegenerate}
Assume $\Psi$ of the form (\ref{sub}). 
\begin{itemize}
\item[i)] If $\textbf{(A)}$ is satisfied then for all $x\geq 0$, $\mathbb{P}_x(\underset{t\rightarrow \infty}{\lim} U^0_{t}=0)=1$.
\item[ii)] If $\textbf{(A)}$ is not satisfied then for all $x\geq 0$, 
$$\mathbb{P}_x(\underset{t\rightarrow \infty}{\lim} U^0_{t}=0)=
\frac{\int_{x}^{\infty}\exp\left(\int_{\theta}^{y}\frac{2\Psi(z)}{cz}\ddr z\right)\ddr y}{\int_{0}^{\infty}\exp\left(\int_{\theta}^{y}\frac{2\Psi(z)}{cz}\ddr z\right)\ddr y}.$$
\end{itemize}
\end{lemma}

\begin{proof}
We show i) and ii) simultaneously. By Lemma \ref{absorbing0}, $|s(0)|<\infty$ and if $s(\infty)<\infty$, then \cite[Theorem 3.3 page 220]{MR1398879} ensures $$\mathbb{P}_x (U^0_t\underset{t\rightarrow +\infty}{ \longrightarrow} 0)=\frac{s(\infty)-s(x)}{s(\infty)-s(0)}.$$
Moreover if $s(\infty)=\infty$ then $\mathbb{P}_x (U^0_t\underset{t\rightarrow +\infty}{ \longrightarrow} 0)=1$. We show that $s(\infty)=\infty$ if and only if $(A)$ is satisfied. One has $s(\infty)=\int_{\theta}^{\infty}\exp\left(\int_{\theta}^{y}\frac{2\Psi(z)}{cz}\ddr z\right)\ddr y$. Note that $\frac{\Psi(z)}{z}=-\frac{\lambda}{z}-\delta-\int_{0}^{\infty}e^{-zx}\bar{\pi}(x)\ddr x.$ We adapt the proof of Lambert \cite[Section 4.1]{MR2134113}. 
We first discuss the case of $\textbf{(A)}$ not satisfied. This corresponds to assume that at least one of the following condition is satisfied:
$$\delta>0 \textit{ or } \delta=0, \bar{\pi}(0)<+\infty \text{ and } \bar{\pi}(0)+\lambda>c/2 \textit{ or } \delta=0 \text{ and } \bar{\pi}(0)=+\infty .$$
\begin{itemize}
\item Assume $\delta>0$. Since $\frac{\Psi(z)}{z} \underset{z\rightarrow +\infty}{\longrightarrow} -\delta$, then there exists $y_0$ such that for all $z\geq y_0$, $\frac{2\Psi(z)}{cz}\leq -\frac{2\delta}{c}+\frac{\delta}{c}= -\frac{\delta}{c}<0$. Therefore
\begin{align*}
s(\infty)=\int_{\theta}^{\infty}\exp\left(\int_{\theta}^{y_0}\frac{2\Psi(z)}{cz}\ddr z+\int_{y_0}^{y}\frac{2\Psi(z)}{cz}\ddr z\right)&\leq c_{y_0}\int_{\theta}^{\infty}\exp\left(-\int_{y_0}^{y}\frac{\delta}{c}\ddr z\right)\ddr y\\
&=c_{y_{0}}e^{\frac{\delta}{c}y_0}\int_{\theta}^{\infty}e^{-\frac{\delta}{c}y}\ddr y<\infty \end{align*}
where $c_{y_0}=\exp\left(\int_{\theta}^{y_0}\frac{2\Psi(z)}{cz}\ddr z\right)$.
\item Assume $\delta=0$, $\bar{\pi}(0)<+\infty$ and $\bar{\pi}(0)+\lambda>c/2$. By Tonelli, one has 
\begin{align}\label{tonelli}-\int_{\theta}^{y}\frac{2\Psi(z)}{cz}\ddr z&=\frac{2\lambda}{c}\ln\left(\frac{y}{\theta}\right)+\frac{2}{c}\int_{\theta}^{y}\ddr z\left(\int_{0}^{\infty}e^{-zx}\bar{\pi}(x)\ddr x\right)\\
&=\frac{2\lambda}{c}\ln\left(\frac{y}{\theta}\right)+\frac{2}{c}\int_{0}^{\infty}\frac{e^{-\theta x}-e^{-yx}}{x}\bar{\pi}(x)\ddr x.\end{align}
We will show that 
\begin{equation}\label{ktheta}
k_{\theta}(x):=\underset{y\rightarrow \infty}{\lim} \left[\int_{0}^{\infty}\frac{e^{-\theta z}-e^{-yz}}{z}\bar{\pi}(z)\ddr z-\int_{x/y}^{x}e^{-\theta z}\frac{\bar{\pi}(z)}{z}\ddr z\right] \text{ exists and is finite.}
\end{equation}
Assume for now that (\ref{ktheta}) is proven. Since $\int_{x/y}^{x}e^{-\theta z}\frac{\bar{\pi}(z)}{cz}\ddr z\geq \frac{1}{c}\bar{\pi}(x)e^{-\theta x}\ln (y)$, then there exists $y_0$ such that if $y\geq y_0$, then
\begin{align*}-\frac{2}{c}\int_{\theta}^{y}\frac{\Psi(z)}{z}\ddr z&\geq  \frac{2}{c}\left(k_{\theta}(x)-1-\ln(\theta)+\lambda\ln(y)+\int_{x/y}^{x}e^{-\theta z}\frac{\bar{\pi}(z)}{z}\ddr z\right)\\
&\geq c_{\theta}(x)+\frac{2}{c}(\lambda+\bar{\pi}(x)e^{-\theta x})\ln y.
\end{align*}
Choose $x$ close enough to $0$ such that  $\lambda+\bar{\pi}(x)e^{-\theta x}>c/2$. Therefore
\begin{align*}
s(\infty)=\int_{\theta}^{\infty}\exp\left(\int_{\theta}^{y}\frac{2\Psi(z)}{cz}\ddr z\right)\ddr y&\leq K_{\theta,y_0}+ \int_{y_0}^{\infty}\exp\left(c_{\theta}(x)-\frac{2}{c}(\lambda+\bar{\pi}(x)e^{-\theta x})\ln y \right)\ddr y\\
&=K_{\theta,y_0}+K_{\theta,x}\int_{\theta}^{\infty}y^{-\frac{2}{c}(\lambda+\bar{\pi}(x)e^{-\theta x})}\ddr y<\infty
\end{align*}
where $K_{\theta,y_0}$ and $K_{\theta,x}$ are some constants.\\
\item Assume that $\bar{\pi}(0)=+\infty$ and define $\pi_{b}(\ddr x)$ by $\bar{\pi}_b(x)=\bar{\pi}(x\vee b)$. For all $b>0$ and all $x\geq 0$, $\bar{\pi}_{b}(x)\leq \bar{\pi}(x)$. The measure $\pi_b$ is finite and one can choose $b$ such that $\bar{\pi}_b(0)>c$. One has for $y$ large enough: $$\int_{0}^{\infty}\frac{e^{-\theta x}-e^{-yx}}{x}\bar{\pi}(x)\ddr x\geq \int_{0}^{\infty}\frac{e^{-\theta x}-e^{-yx}}{x}\bar{\pi}_b(x)\ddr x$$ 
Applying the same argument as above with $\lambda=0$ and $\pi_b$ instead of $\pi$, one obtains $s(\infty)<\infty$.\\
\end{itemize}
Assume now that condition \textbf{(A)} hold. Namely, $\delta=0$, $\bar{\pi}(0)<+\infty$ and $\bar{\pi}(0)+\lambda\leq c/2$. By using (\ref{ktheta}), there exists $y_0$ such that for $y\geq y_0$, $$-\int_{\theta}^{y}\frac{2\Psi(z)}{cz}\ddr z\leq k_{\theta}(x)+1+\frac{2}{c}\ln\left(\frac{y}{\theta}\right)+
\frac{2}{c}\int_{x/y}^{x}e^{-\theta z}\frac{\bar{\pi}(z)}{z}\ddr z\leq k_{\theta}(x)+1-\frac{2}{c}\ln(\theta)+\frac{2}{c}(\lambda+\bar{\pi}(0))\ln (y)$$
Since $\bar{\pi}(0)+\lambda\leq c/2$, then $$s(\infty)\geq K_{\theta,y_0}+K_{\theta,x}\int_{\theta}^{\infty}y^{-2(\bar{\pi}(0)+\lambda)/c}\ddr y=\infty.$$
We show now (\ref{ktheta}):
\begin{align*}
&\int_{0}^{\infty}\left(e^{-\theta z}-e^{-yz}\right)\frac{\bar{\pi}(z)}{z}\ddr z-\int_{x/y}^{x}e^{-\theta z}\frac{\bar{\pi}(z)}{z}\ddr z\\
&=\int_{0}^{x/y}\left(e^{-\theta z}-e^{-yz}\right)\frac{\bar{\pi}(z)}{z}\ddr z-\int_{x/y}^{x}e^{-yz}\frac{\bar{\pi}(z)}{z}\ddr z+\int_{x}^{\infty}\left(e^{-\theta z}-e^{-yz}\right)\frac{\bar{\pi}(z)}{z}\ddr z\\
&=\underbrace{\int_{0}^{x/y}\left(e^{-\theta z}-e^{-yz}\right)\frac{\bar{\pi}(z)}{z}\ddr z}_{I_1(x,y)}-\underbrace{\int_{x}^{xy}e^{-u}\frac{\bar{\pi}(u/y)}{u}\ddr u}_{I_2(x,y)}+\underbrace{\int_{x}^{\infty}\left(e^{-\theta z}-e^{-yz}\right)\frac{\bar{\pi}(z)}{z}\ddr z}_{I_3(x,y)}.
\end{align*}
By changing variable, one has $I_{1}(x,y)=\int_{0}^{x}\left(e^{-\theta \frac{u}{y}}-e^{-u}\right)\frac{\bar{\pi}(u/y)}{u}\ddr u$. Since $y\mapsto \bar{\pi}(u/y)$ is non-decreasing and converges to $\bar{\pi}(0)$ as $y$ goes to $\infty$, one can apply the monotone convergence theorem, this provides \[I_{1}(x,y)\underset{y\rightarrow +\infty}{\longrightarrow}\bar{\pi}(0)\int_{0}^{x}(1-e^{-u})\frac{\ddr u}{u}<\infty.\] The monotone convergence theorem readily applies to $I_{2}(x,y)$: 
\[I_{2}(x,y)\underset{y\rightarrow +\infty}{\longrightarrow} \bar{\pi}(0)\int_{x}^{\infty}\frac{e^{-u}}{u}\ddr u<\infty.\]
By Lebesgue's theorem, $I_{3}(x,y)=\int_{x}^{\infty}\left(e^{-\theta z}-e^{-yz}\right)\frac{\bar{\pi}(z)}{z}\ddr z\underset{y\rightarrow \infty}{\longrightarrow} \int_{x}^{\infty}\frac{e^{-\theta z}\bar{\pi}(z)}{z}\ddr z<\infty$. Finally,
$$k_{\theta}(x)=\bar{\pi}(0)\left(\int_{0}^{x}\frac{1-e^{-u}}{u}\ddr u+\int_{x}^{\infty}\frac{e^{-u}}{u}\ddr u\right)+\int_{x}^{\infty}e^{-\theta z}\frac{\bar{\pi}(z)}{z}\ddr z<\infty.$$
\end{proof}
We end this appendix by studying the Laplace transform of the first entrance times of Ornstein-Uhlenbeck-type processes. Recall $(R_t,t\geq 0)$ the Ornstein-Uhlenbeck type process with parameters $\Psi$ and $c/2$ as introduced in Section \ref{existence}. We establish the formula (\ref{LaplacetransformhittingOU}) of the Laplace transform of $\sigma_a:=\inf\{t\geq 0, R_t\leq a\}$, needed in Lemma \ref{longtermminimal} and Lemma \ref{totalprogeny}. 
\begin{lemma}\label{hittingOU} Assume that $\Psi(z)\geq 0$ for some $z\geq 0$. For any $a\geq 0$ and $\mu>0$ 
\begin{equation*}\mathbb{E}_{z}[e^{-\mu \sigma_a}]=\frac{\int_{0}^{\infty} x^{\mu-1}e^{-zx-\int_{\theta}^{x}\frac{2\Psi(y)}{cy}\ddr y}\ddr x}{\int_{0}^{\infty}x^{\mu-1}e^{-ax-\int_{\theta}^{x}\frac{2\Psi(y)}{cy}\ddr y}\ddr x}.
\end{equation*}
\end{lemma}
\begin{proof}
For any $\mu>0$, define $g_{\mu}(x)=x^{\mu-1}e^{-\int_{\theta}^{x}\frac{2\Psi(y)}{cy}\ddr y}$. The function $g_{\mu}$ solves the following equation $$\left(\Psi(x)-\mu+\frac{c}{2}\right)g_\mu(x)+\frac{c}{2}xg'_\mu(x)=0 \text{ for all } x\geq 0.$$ We check now that $\int_{0}^{\infty}g_{\mu}(x)\ddr x<\infty$. Note that either $\int_{0}^{\theta}\frac{\Psi(u)}{u}\ddr u=-\infty$ or $\int_{0}^{\theta}\frac{\Psi(u)}{u}\ddr u\in (-\infty,\infty)$. In both cases, for any $b>0$ there is a constant $C>0$ such that
$\int_{0}^{b}g_\mu(x)\ddr x\leq C\int_{0}^{b}x^{\mu-1}\ddr x<\infty$ since $\mu>0$. By assumption $\Psi$ is not the Laplace exponent of a subordinator, so there exists $a\in (0,\infty)$ such that for all $u\geq a$, $\Psi(u)\geq \Psi(a)>\frac{2\mu}{c}$. Then, for some other constant $C$,
\begin{align*}\int_{a}^{\infty}g_{\mu}(x)\ddr x&=C\int_{a}^{\infty}x^{\mu-1}e^{-\int_{a}^{x}\frac{2\Psi(u)}{cu}\ddr u}\ddr x\leq C\int_{a}^{\infty}x^{\mu-1}e^{-\int_{a}^{x}\frac{2\Psi(a)}{cu}\ddr u}\ddr x\\
&\leq C\int_{a}^{\infty}x^{-(\frac{2\Psi(a)}{c}-\mu)-1} \ddr x<\infty.
\end{align*}
Set $f_\mu(z)=\int_{0}^{\infty}e^{-xz}g_{\mu}(x)\ddr x$ for any $z\geq 0$. This is a $C_0^{2}$ decreasing function. For any $z\in (0,\infty)$ and $x\in (0,\infty)$, $\mathcal{L}^{R}e_x(z)=\Psi(x)e_{x}(z)+\frac{c}{2}xze_x(z).$  We now verify that $f_\mu$ is an eigenfunction of the generator $\mathcal{L}^{R}$: \begin{align*}
&\mathcal{L}^{R}f_\mu(z)-\mu f_{\mu}(z)\\
&=\int_{0}^{\infty}g_{\mu}(x)\left(\mathcal{L}^{R}e_{x}(z)-\mu e_{x}(z)\right)\ddr x=  \int_{0}^{\infty}g_{\mu}(x)
(\Psi(x)+\frac{c}{2}xz-\mu)e_{x}(z)\ddr x\\
&=\int_{0}^{\infty}(\Psi(x)-\mu)e^{-xz}g_{\mu}(x)\ddr x+\int_{0}^{\infty}\frac{c}{2}ze^{-xz}xg_\mu(x)\ddr x\\
&=\int_{0}^{\infty}(\Psi(x)-\mu)e^{-xz}g_{\mu}(x)\ddr x+\frac{c}{2}\left(\left[-e^{-xz}xg_\mu(x)\right]_{x=0}^{x=\infty}+\int_{0}^{\infty}e^{-xz}(g_\mu(x)+xg'_\mu(x))\ddr x\right)\\
&=\int_{0}^{\infty}\left((\Psi(x)-\mu+\frac{c}{2})g_\mu(x)+\frac{c}{2}xg'_\mu(x)\right)e^{-xz}\ddr x=0.
\end{align*}
By It\^o's formula (see e.g. \cite[Lemma 7]{MR2128632} for a similar calculation), the process $(e^{-\mu t}f_{\mu}(R_t),t\geq 0)$ is a local martingale. Since $(R_t,t\geq 0)$ has no negative jumps, and the function $f_\mu$ is decreasing, one has for any $t\leq \sigma_a$, $R_t\geq a$ and $f_\mu(R_t)\leq f_\mu(a)$ $\mathbb{P}_z$-a.s, for all $z\geq a$. Therefore, $(e^{-\mu (t\wedge \sigma_a)}f_{\mu}(R_{t\wedge \sigma_a}),t\geq 0)$ is a bounded martingale, and  by the optional stopping theorem, one get
$$\mathbb{E}_{z}[e^{-\mu \sigma_a}]=\frac{f_{\mu}(z)}{f_\mu(a)}.$$
\end{proof}

\textbf{Acknowledgements:} The author would like to thank Ger\'onimo Uribe-Bravo for a precious help in the proof of Lemma \ref{dualnonexplosive} and thanks Adrian Gonzales-Casanova for many helpful discussions. This work is partially  supported by the French National Research Agency (ANR): ANR GRAAL (ANR-14-CE25-0014) and by LABEX MME-DII (ANR11-LBX-0023-01). 

\bibliographystyle{amsalpha}

\end{document}